\documentclass[a4paper]{article}
\usepackage[english]{babel}
\usepackage[utf8]{inputenc}
\usepackage{dsfont}
\usepackage{graphicx}
\usepackage{color}
\usepackage{epsfig}
\usepackage{cellspace}
\usepackage[normalem]{ulem}
\usepackage{diagbox}
\usepackage{amsthm,amssymb,amsbsy,amsmath,amsfonts,amssymb,amscd,mathrsfs}
\usepackage{colortbl}
\usepackage{verbatim}
\usepackage{url}
\usepackage{enumerate}
\usepackage{bbm}
\newtheorem{rem}{Remark}[section]
\newtheorem{prop}{Proposition}[section]
\newtheorem{propri}{Property}[section]
\newtheorem{defi}{Definition}[section]
\newtheorem{lem}{Lemma}[section]

\newtheorem{thm}{Theorem}[section]

\def\ds{\displaystyle}
\newcommand{\vphi}{\varphi}

\newcommand{\eps}{\varepsilon}
\newcommand{\R}{\mathbb{R}}

\newcommand\be{\begin{equation}}
\newcommand\ee{\end{equation}}
\numberwithin{equation}{section}

\title{Stability of the chemostat system with a mutation factor}
\date{\today}
\author{T. Bayen\thanks{Avignon Universit\'e, Laboratoire de Math\'ematiques d'Avignon (EA 2151), F-84018 Avignon, France. {\tt terence.bayen@univ-avignon.fr}} , H. Cazenave-Lacroutz\thanks{Avignon Universit\'e, Laboratoire de Math\'ematiques d'Avignon (EA 2151), F-84018 Avignon, France. {\tt henri.cazenave-lacroutz@univ-avignon.fr}} , J. Coville\thanks{UR 546 Biostatistique et Processus Spatiaux, INRAE, Domaine St Paul Site Agroparc, F-84000 Avignon, France. 
{\tt jerome.coville@inrae.fr}}}

\begin{document}
\maketitle
\begin{abstract} In this paper, we consider a resource-consumer model taking into account a mutation effect between species (with constant mutation rate). 
The corresponding mutation operator is a discretization of the Laplacian in such a way 
that the resulting dynamical system can be viewed as a regular perturbation of the classical chemostat system. 
We prove the existence of a unique locally stable steady-state for every value of the mutation rate and every value of the dilution rate not exceeding a critical value. 
In addition, we give an expansion of the steady-state in terms of the mutation rate and we prove a uniform persistence property of the dynamics related to each species. 
Finally, we show that this equilibrium is globally asymptotically stable for every value of the mutation rate 
provided that the dilution rate is with small enough values. 
\end{abstract}
\section{Introduction}
The chemostat system was introduced in the fifties to model the behavior of bacteria competing for a same substrate (see \cite{Monod1,Monod2,szilard}). 
It has now become a reference model for the modeling of ecosystems (lakes, rivers, microalgae,...), see, {\it{e.g.}}, \cite{GG66}, 
and it is  widely used in biotechnology, for instance, for the control of the production of microalgae of interest or in waste water treatment (see, {\it{e.g.}}, \cite{Bastin90,dochain1,BHM,BMMS} and references herein). The chemostat system with $n$ species competing for one same resource writes 
\begin{equation}{\label{sys0}}
\left|
\begin{array}{cl}
\dot{x}_i&=(\mu_i(s)-D)x_i, \quad  \quad 1 \leq i \leq n,\vspace{0.1cm} \\
\dot{s}&=\ds -\sum_{j=1}^n \frac{\mu_j(s)}{Y_i}x_j+D(s_{in}-s), 
\end{array}
\right.
\end{equation}
where $x_i$ is the concentration of species $i$ (the consumers) and $s$ denotes the substrate  concentration (the resource).  The  numbers $Y_i$ are the yield coefficients, the parameter $s_{in}$ is the input substrate coefficient, the functions $\mu_i$ are the kinetics, and $D$ is the dilution rate.  
Properties of \eqref{sys0} has been studied a lot \cite{dochain1,SmithWalt,floc,GMR2009,alain-livre,Hsu1,Hsu2,lobry,Sari1,Wolko1}, 
and one essential feature is the famous {\it{competitive exclusion principle}} (CEP) which 
 asserts that, asymptotically, only one species survives \cite{Hsu1,Wolko1,SmithWalt,alain-livre}. 
Many extensions of the CEP have been studied in presence of delay, external inhibitors, or variable yields 
(see, {\it{e.g.}}, \cite{GMR2009,Hsu2,Sari1} among others). 
It is also worth mentioning that the CEP predicts  exclusion of the less competitive species and not coexistence in contrast with observations in several ecosystems. 
That is why, extensions of the chemostat system were also developed (such as in \cite{floc}) to cope with this reality.
In this paper, we consider another extension of the chemostat system related to  the possibility 
for a species to produce mutants or to appear through mutation 
(see, {\it{e.g.}}, \cite{szilard,deleen,deleen2}). It turns out that mutation will modify the behavior of the system leading to coexistence. 
There exist various approaches to model this phenomenon: each species may convert into other species with a mutation rate depending on various parameters such as the kinetics (see, {\it{e.g.}}, \cite{lobry} or \cite{Coralie}). Throughout this paper, we shall assume that the dispersion is such that each species $i$ converts into neighbor species $i+1$ and $i-1$ with a constant mutation rate. This amounts to add a linear term $Tx$ in the sub-system satisfied by the concentration vector $x$ in \eqref{sys0}, where $T$ is the mutation matrix. 
Our objective in this paper is to provide a thorough study of asymptotic stability properties of the resulting system. 
Surprisingly, to our best knowledge, few papers addressed this question apart \cite{Arkin,deleen} and \cite{BM,BM2} which study a minimal time  control problem to select optimally species of interest (see also \cite{masci08}). 

Let us  give a quick overview of \cite{deleen} that introduced  
the chemostat system with a mutation. The main result is a global stability property of the coexistence steady-state 
provided that the kinetics are sufficiently close to a nominal one $\mu_0$ as 
well as yield coefficients which also should be close to a nominal value $Y_0$. This means that the quantities   
$\|\mu_i-\mu_0\|_{L^\infty}$ and $|Y_i-Y_0|$ should be small enough for every $1 \leq i \leq n$ to ensure the global stability property. 
This result (in the spirit of \cite{Arkin}) is interesting in itself but it does not predict the behavior of the system whenever 
 kinetics  are not necessarily close to a common one. 
 In this paper, we consider the more general situation where kinetics are of Monod type, but not necessary close to a nominal one. 
Our aim is to address stability properties of the corresponding system.  
Based on experimental studies (see, {\it{e.g.}}, \cite{szilard}, we shall assume that the yield coefficients are equal to one.
As in \cite{deleen}, we shall see that mutation implies coexistence in contrast with the CEP for the chemostat model. 

The paper is structured as follows. In Section \ref{sec-1}, we introduce the chemostat system with mutation and we recall the CEP.  
In Section \ref{sec-2}, we show in Proposition \ref{disjonction} that there is exactly one locally asymptotically stable (LAS) steady-state provided that the dilution rate does not exceed a certain value 
(for which extinction of species would occur). This result extends the analysis of \cite{deleen} and relies on eigenvalue properties of a rank one perturbation of a symmetric non-positive matrix (see \cite{Bierkens}). 
In Section \ref{sec-dev-lim}, we compute an expansion of the steady-state in terms of the mutation factor. We obtain that way an interesting result asserting that, at steady-state, few 
species dominate, namely the one that wins the selection in absence of mutation, and its neighbors (see, Proposition \ref{prop-expan}). 
We also study the converse case, {\it{i.e.}}, when the mutation factor becomes large (w.r.t.~the kinetics of the system). 
In Section \ref{sec-persist}, we show a {\it{uniform persistence property}} (see \cite{ST11}). 
This property asserts that, asymptotically, each species is present in the system (and not only the total biomass \cite{deleen}). 
This uniform persistence property highlights the difference of the chemostat system with mutation w.r.t.~the classical chemostat system (leading to exclusion of every less competitive species). 
In Section \ref{sec-4}, we give our main result (Theorem \ref{GAS-thm}) about global stability of the steady-state for small enough dilution rates. 
To show that the equilibrium is GAS, we proceed  in three steps. First, we study stability properties of the system without dilution rate (with mutation). 
Next, we show that the GAS property is valid on an invariant attractive manifold associated with the system for small enough values of the dilution rate. 
This requires to prove a robust  persistence property (in line with \cite{deleen}) and to use perturbation results of \cite{SW99} (see also \cite{Salceanu,SW99,Thieme93}). We conclude by using 
 the theory of asymptotically autonomous systems (see, {\it{e.g.}}, \cite{Thieme92}). 
%
\section{Recap on the chemostat model and preliminary properties}{\label{sec-1}}
Throughout this paper, we consider a chemostat system with $n\geq 1$ species including a mutation effect between species. We suppose that each species $i$ is able to convert into 
species $i+1$ and $i-1$ with a constant mutation rate. This yields the following dynamical system
\begin{equation}{\label{sys1}}
\left|
\begin{array}{cl}
\dot{x}_i&=(\mu_i(s)-u)x_i + \eps (Tx)_i, \quad  \quad 1 \leq i \leq n,\vspace{0.1cm} \\
\dot{s}&=\ds -\sum_{j=1}^n \mu_j(s)x_j+u(1-s),
\end{array}
\right.
\end{equation}
where: 
\begin{itemize}
\item[$\bullet$] For every $1 \leq i \leq n$, $x_i$ denotes the concentration of species $i$ and $s$  the substrate concentration. 
\item[$\bullet$] For every $1 \leq i \leq n$, the kinetics $\mu_i$ of species $i$ is supposed to be of Monod type, {\it{i.e.}}, $\mu_i(s)=\frac{m_is}{a_i+s}$ ($a_i$, $m_i$ are positive numbers such that $i\not=j \; \Rightarrow (m_i,a_i)\not=(m_j,a_j)$). 
\item[$\bullet$] The input substrate concentration has been renormalized to $1$ and the dilution rate is $u\in \R_+$. 
\item[$\bullet$] The mutation parameter is $\eps \geq 0$, $x:=(x_1,...,x_n)^\top$ denotes the column vector of the species concentrations (the symbol $^\top$ is the transpose operator), and the mutation matrix\footnote{As usual, matrices are named using capital letters and coefficients are represented by lower case letters.} $T\in \R^{n \times n}$ is:
{\small{
\begin{equation}{\label{mut-matrix}}
T:=\left[
\begin{array}{ccccc}
-1 & 1 & 0 & \cdots & 0\\
1 & -2 & 1 & \cdots & 0\\
\vdots & \ddots & \ddots & \ddots & \vdots\\
0 & \cdots & 1 &  -2 & 1\\
0& \cdots & 0& 1& -1
\end{array}
\right].
\end{equation}
}}
\end{itemize}
Note that the symmetric matrix $T$ corresponds to the discretization of the one-dimensional Laplace equation (Poisson problem) with Neumann boundary conditions. 
We recall that it is {\it{quasi-positive}} ({\it{i.e.}}, for $i\not=j$, $t_{i,j}\geq 0$) and {\it{irreducible}} (since $(T+rI_n)^k$ is with positive entries for $r=3$ and  $k$ large enough).  
From Perron-Frobenius's Theorem (see, {\it{e.g.}}, \cite{Berman}), the largest eigenvalue of $T$ (called the {\it{Perron root}}) is simple and the {\it{Perron vector}} ({\it{i.e.}}, the corresponding unitary eigenvector) is positive. Finally, 
the sum of the coefficients of $T$ on a row is always zero\footnote{This property is also essential for proving the invariance of the set $\Delta$ (Lemma \ref{invaLem}).}, 
so, $0$ is necessarily the Perron root of $T$ and $a/\sqrt{n}$ is the Perron vector where 
$a:=(1,...,1)$. Note also that $T$ is non-positive. 
\begin{rem} 
More complex mutation terms between species can be also considered in the chemostat model as for instance in {\rm{\cite{Arkin,lobry}}} where the mutation factor involves the kinetics of the species. 
Mutation could also involve a pool of species close to some index $i$ (not only the two closest indexes of $i$), 
but, in this paper, we restrict our attention to a mutation term $\eps Tx$ where $T$ is given by \eqref{mut-matrix} (see {\rm{\cite{deleen}}}) and $\eps\geq 0$ is eventually a small parameter. System \eqref{sys1} can be viewed as an approximation of a population dynamics model involving a phenotypic trait, see, {\it{e.g.}}, {\rm{\cite{diekmann04,diekmann05,mirrahimi12,Perthame}}} (among others). 
\end{rem}
When $\eps=0$, we retrieve the classical chemostat model with $n \geq 1$ species described by the system
\begin{equation}{\label{chem1}}
\left|
\begin{array}{cl}
\dot{x}_i&=(\mu_i(s)-u)x_i,  \quad \quad 1 \leq i \leq n,\vspace{0.1cm} \\
\dot{s}&=\ds -\sum_{j=1}^n \mu_j(s)x_j+u(1-s),
\end{array}
\right.
\end{equation}
in such a way that \eqref{sys1} can be viewed as a regular perturbation\footnote{For the concept of regular perturbation of a dynamical system, we refer to 
\cite{Atha82,SW99,Magal} (see also references herein).} of \eqref{chem1}  for small values of $\eps$. 
When dealing with the chemostat system, it is usual to introduce the so-called {\it{break-even concentrations}} $\lambda_i(u)\in [0,+\infty)$ 
that play a key role in the chemostat system: 
\begin{equation}{\label{bec}}
    \lambda_i(u):=
\left\{
\begin{array}{lll}
\mu_i^{-1}(u) & \mathrm{if} \; \; \mu_i(1)>u, & \vspace{0.1cm}\\
+\infty & \mathrm{otherwise},  & 
\end{array}
\right. \quad 1 \leq i \leq n. 
\end{equation}
In order to study asymptotic stability properties of \eqref{sys1}, it will be helpful to recall the global stability properties of \eqref{chem1}. 
Doing so, set, 
$$
E_i=(0,...,0,1-\lambda_i(u),0,...,0,\lambda_i(u))\in \R^{n+1},
$$
for $1 \leq i \leq n$ and observe that $E_i$ is a steady-state of \eqref{chem1} provided that $\lambda_i(u)<+\infty$. In addition, the point 
$$
E_{wo}:=(0,...,0,1) \in \R^{n+1},
$$ 
is also an equilibrium of \eqref{chem1} (called {\it{washout steady-state}}). Thus, \eqref{chem1} has at most $n+1$ steady-states. 
The well-known {\it{competitive exclusion principle}} (CEP) can be now stated as follows. 
\begin{thm} {\label{CEP-thm}}
$\mathrm{(i)}$. Let $u>0$. If there is a unique $1 \leq i_0 \leq n$ such that $\lambda_{i_0}(u)=\min_{1 \leq i \leq n} \lambda_i(u)<+\infty$, then, for every initial condition $(x^0,s^0)\in [0,+\infty)^n \times [0,1]$ such that $x^0_{i_0}>0$, the unique solution of \eqref{chem1} starting at $(x^0,s^0)$ at time $0$ converges to  $E_{i_0}$. 

\noindent $\mathrm{(ii)}$. Let $u>0$. If $\min_{1 \leq i \leq n} \lambda_i(u)=+\infty$, then,  for every initial condition $(x^0,s^0)\in [0,+\infty)^n \times [0,1]$, the unique solution of \eqref{chem1} starting at $(x^0,s^0)$ at time $0$ converges to $E_{wo}$.  
\end{thm}
\begin{rem} In case (i) of the previous theorem, if the minimum is non-unique, then, coexistence may occur {\rm{\cite{SmithWalt,alain-livre}}}, but, 
we do not develop this point here because it is non-generic. Throughout the paper, we shall assume (if necessary) that $u$ is such that the minimum is unique.  
\end{rem}
The competitive exclusion principle asserts a global stability property of one species for \eqref{sys1} initially present in the vessel, {\it{i.e.}}, only one species survives generically (namely the one with the least break-even concentration). 
\begin{rem}{\label{lyapu-rem}}
There are various proofs of this result (see, {\it{e.g.}}, {\rm{\cite{alain-livre,SmithWalt,RV}}} among others). When kinetics are of Monod type, a direct way is to use a Lyapunov function. Doing so, write 
$E_{i_0}=(x_1^*,...,x_n^*,s^*)$ and \eqref{chem1} as
\begin{equation}{\label{chem2}}
\left|
\begin{array}{cl}
\dot{\tilde x}_i&=(\mu_i(s^*)-u)\tilde x_i +  (\mu_i(s)-\mu_i(s^*))x_i, \quad \quad 1 \leq i \leq n, \vspace{0.1cm}\\
\dot{\tilde s}&= -u \tilde s - \sum_{j=1}^n \mu_j(s)\tilde x_j - \sum_{j=1}^n (\mu_j(s)-\mu_j(s^*)) x_j^*, 
\end{array}
\right.
\end{equation}
where $\tilde x_i:=x_i-x_i^*$ and $\tilde s:=s-s^*$. Next, it can be verified that the function
\begin{equation}{\label{lyap1}}
V(\tilde x,\tilde s):=\phi_{s^*}(\tilde s)+\sum_{j=1}^n \frac{a_j+s^*}{a_j}\phi_{x_j^*}(\tilde x_j)+\frac{1}{2}\Big{[}\tilde s + \sum_{j=1}^n \tilde x_j\Big{]}^2,
\end{equation}
is a strict Lyapunov function for \eqref{chem2} where $\phi_a(\sigma):=\sigma-a\ln(1+\sigma/a)$, $a>0$, see, {\it{e.g.}}, {\rm{\cite{GMR2009,Hsu1}}} and references herein. 
However, even if \eqref{sys1} is a regular perturbation of \eqref{chem1}, it is an open question how to construct a Lyapunov function for \eqref{sys1} based on \eqref{lyap1}  
(see {\rm{\cite{GMR2009,Hsu1}}}) or on relative entropy identities (see {\rm{\cite{coville1,coville2}}}). 
Besides, global stability property may fail to hold under small perturbations of a dynamical 
system\footnote{As an example, consider the system $\dot{x}=-x/(1+x^2)+\eps x$ for which $0$ is GAS for $\eps=0$ and LAS for every $\eps\in [0,1)$. But $0$ is never GAS for every $\eps>0$. We thank F. Mazenc for indicating to us such an example.}. 
\end{rem}
Going back to \eqref{sys1}, observe that solutions to \eqref{sys1} are defined globally over $\R_+$ and that the dynamics of $x$ can be rewritten 
\begin{equation}{\label{dyn-x}}
\dot{x}=B(s,u,\eps)x, 
\end{equation}
where 
\begin{equation}{\label{Bs}}
B(s,u,\eps):=M(s)-u I_n+\eps T\in \R^{n \times n},
\end{equation}
$I_n\in \R^{n\times n}$ denotes the identity matrix, and $M(s)$ stands for the diagonal matrix 
$M(s):=\mathrm{diag}(\mu_1(s),...,\mu_n(s))$.   
Note that the matrix $B(s,u,\eps)$ 
is quasi-positive for every $s\in \R$, so 
$\R_+^{n}$ is forward invariant by \eqref{dyn-x} (see, {\it{e.g.}}, \cite{deleen}). 
In contrast with \eqref{chem1}, it is enough to 
suppose that only one species is present at time $0$ to ensure that for every time $t>0$, one has $x_i(t)>0$ for every $1 \leq i \leq n$, as we now show.  
\begin{propri}{\label{propri-pos}}
Let $(\eps,u)\in \R_+^* \times \R_+$ and let $x(\cdot)$ be a solution to \eqref{sys1}. If there is $1 \leq i \leq n$ such that $x_i(0)>0$, then, for every time $t>0$, one has $x_j(t)>0$ for every $1 \leq j \leq n$. 
\end{propri}
\begin{proof} Recall that 
$\R_+^n$ is forward invariant by \eqref{dyn-x}. 
We  claim that for every time $t\geq 0$, one has $x_i(t)>0$. Indeed, let $t_0:=\inf\{t > 0 \; ; \; x_i(t)=0\}$ and suppose that $t_0<+\infty$. Since $x_i(t)>0$ for $t\in [0,t_0)$, one has
$$
x_i(t_0)=0 \quad \mathrm{and} \quad \dot{x_i}(t_0)=\eps(x_{i+1}(t_0)+x_{i-1}(t_0))\leq 0, 
$$
which implies $x_{i+1}(t_0)=x_{i-1}(t_0)=0$ 
Observe now that $\dot{x}_{i+1}(t_0)=\eps x_{i+2}(t_0)$ and since $x_{i+1}(\cdot)$ vanishes at $t=t_0$, we deduce that 
$$
\dot{x}_{i+1}(t_0)=\eps x_{i+2}(t_0)\leq 0. 
$$
Thus, one must have $x_{i+2}(t_0)=0$. In the same way, we  get that $x_{i-2}(t_0)=0$. By induction over $j$, we deduce that for every $1 \leq j \leq n$, one has $x_j(t_0)=0$. 
By Cauchy-Lipschitz's Theorem, one must have $x\equiv 0$ over $\R_+$ which is a contradiction since $x_i(0)>0$. This proves our claim. 

Let us now show that $x_{i+1}$ never vanishes over $(0,+\infty)$. If there is $t_1>0$ such that $x_{i+1}(t_1)=0$, then, we would have $\dot{x}_{i+1}(t_1) \leq 0$, thus
$$
\dot{x}_{i+1}(t_1)=\eps(x_i(t_1)+x_{i+2}(t_1))>0,
$$
since $x_i$ is positive over $\R_+$. This is a contradiction, therefore, one must have $x_{i+1}(t)>0$ for every time $t>0$. We can  repeat this argument step by step 
for every species, which proves the desired property. 
\end{proof}
Note also that if $s(0)\in [0,1]$, then one has $s(t)\in [0,1]$ for every $t\geq 0$. Hence, we shall consider initial conditions in the set
$$
\mathcal{D}:=(\R_+^n\backslash \{0\})  \times [0,1], 
$$
when dealing with \eqref{sys1}. The next property is related to the quantity $$
b:=s+\sum_{j=1}^n  x_j,
$$
and it is crucial in the rest of the paper. 
\begin{lem}{\label{invaLem}} For every $(\eps,u)\in \R_+ \times \R_+^*$, the set 
\begin{equation}
\Delta:=\Big{\{}(x,s)\in \mathcal{D} \; ; \; \sum_{j=1}^n x_j+s=1\Big{\}},
\end{equation}
is an invariant and attractive manifold for \eqref{sys1}. 
\end{lem}
\begin{proof}
From \eqref{sys1}, $b$ satisfies 
$\dot{b}=u(1-b)$, hence $b(t)=1+(b(0)-1)e^{-tu}$ for $t\geq 0$, 
whence the result. 
\end{proof}
This lemma makes possible (if necessary) to reduce 
the stability properties of \eqref{sys1} to the system 
$$
\dot{x}=B\Big{(}1-\sum_{j=1}^n x_j,u,\eps\Big{)}x,
$$
obtained from \eqref{sys1} by  considering conditions in 
$\Delta$. 
Note that if $(x^0,s^0)\in \Delta$, then 
$\sum_{j=1}^n x_j^0 \leq 1$, that is why, it is also useful to introduce the set
\begin{equation}{\label{setF}}
\mathcal{D}':=\Big{\{}x\in [0,+\infty)^n \; ; \; \sum_{j=1}^n x_j\leq 1\Big{\}},
\end{equation}
when dealing with initial conditions in $\Delta$. 
The next property is well-known for \eqref{chem1} (see, {\it{e.g.}}, \cite{alain-livre,SmithWalt}) and it remains unchanged for \eqref{sys1}. 
\begin{propri}
For every $u>0$, there is $c_u>0$ such that for every $\eps \geq 0$ and for every initial condition in $\mathcal{D}$, the unique corresponding solution to \eqref{sys1} satisfies:
\begin{equation}{\label{repulsif}}
\liminf_{t\rightarrow +\infty} s(t)\geq c_u. 
\end{equation}
\end{propri}
\begin{proof} Let $\delta_u:=\sup \{s\in [0,1] \; ; \; \max_{1 \leq j \leq n} \mu_j(s)\leq \frac{u}{8}\}$, Since  $\mu_i(0)=0$ for every $1 \leq i \leq n$, $\delta_u$ is well-defined. 
and we can set $c_u:=\min(\delta_u,1/2)$. From Lemma \ref{invaLem}, there is $t_0 \geq 0$ such that $\sum_{j=1}^n x_j(t)\leq 2$ for every $t\geq t_0$.
If now there is $t \geq t_0$ such that $s(t) \leq c_u$, one has
 $$
 \dot{s}(t) =-\sum_{j=1}^n \mu_j(s(t))x_j(t)+u(1-s(t))\geq -\frac{u}{8}\sum_{j=1}^n x_j(t)+\frac{u}{2}\geq -\frac{u}{4}+\frac{u}{2}=\frac{u}{4}.
 $$
From the preceding inequality, $[0,c_u]$ is a repelling set for the dynamics of $s$ which then implies \eqref{repulsif}.  
\end{proof}
\section{Local asymptotic stability}{\label{sec-2}}
\subsection{Existence of a locally stable equilibrium}
Throughout the paper, given a symmetric matrix $A\in \R^{n \times n}$, we denote by $\lambda(A)$ its largest eigenvalue. 
\begin{lem}{\label{lem-croissance}}
Let $(\eps,u)\in \R_+ \times \R_+^*$. Then, 
one has $\lambda(B(0,u,\eps))<0$ and the mapping $s\mapsto \lambda(B(s,u,\eps))$ is increasing over $[0,1]$. 
\end{lem}
\begin{proof}
Observe that $B(0,u,\eps)=-u I_n+\eps T$ which implies  $\lambda(B(0,u,\eps))=-u<0$. Now, thanks to the Perron-Frobenius Theorem \cite{Berman}, for every $s\in [0,1]$, $\lambda(B(s,u,\eps))$ exists and is of multiplicity one. It is also the unique eigenvalue associated with a positive eigenvector. Recall now that given two quasi-positive irreducible and symmetric matrices $C,D\in \R^{n \times n}$ such that $c_{i,j}\leq d_{i,j}$ for every $1 \leq i,j \leq n$ (with a strict inequality for at least one coefficient), one has $\lambda(C)<\lambda(D)$, see \cite{Berman}.  
Since for every $1 \leq i \leq n$, 
$\mu_i$ is increasing, so is $s\mapsto \lambda(B(s,u,\eps))$. 
\end{proof}
Next, we  study the existence of a  locally stable equilibrium point for \eqref{sys1}. 
Doing so, we shall use a result of \cite{Bierkens} about rank one perturbations of a singular $M$-matrix $A=\rho(H) I_n-H \in \R^{n \times n}$  
where $\rho(H)$ denotes the spectral radius of a given matrix $H\in \R^{n\times n}$. Let us  recall the concept of $M$-matrix. 
\begin{defi}
Given $A=(a_{i,j})\in \R^{n\times n}$, we say that $A$ is an $M$-matrix if there exists a matrix $H\in \R^{n\times n}$ and $s>\rho(H)$ such that 
$$
A=sI_n-H \quad \mathrm{and} \quad s>\rho(H). 
$$
\end{defi}
Notice that if $A$ is an $M$-matrix, then its eigenvalues are with nonnegative real parts and $a_{i,j}\leq 0$ for $i\not=j$. 
Theorem 2.7 of \cite{Bierkens} provides sufficient conditions for a matrix $A+v w^\top$ (where $v,w\in \R^n$) to be positive stable if $0$ is a geometrically simple eigenvalue of $A$. We refer to \cite{Bierkens} for the precise statement of those conditions.  
The next Proposition (point (ii) only) extends the analysis of {\rm{\cite{deleen}}} showing that, depending on the values of $(\eps,u)$, 
\eqref{sys1} has a unique locally stable equilibrium. 
The notation $\|\cdot\|$ stands for the euclidean norm in $\R^n$. 
\begin{prop}{\label{disjonction}}
{\rm{(i)}} If $(\eps,u)\in \R_+^* \times \R_+^*$ is such that $\lambda(B(1,u,\eps))\leq 0$, then, the washout steady-state $E_{wo}$ is the only equilibrium of \eqref{sys1} and it is stable. If $\lambda(B(1,u,\eps))<0$, it is globally asymptotically stable. 
\vspace{0.1cm}
\\
{\rm{(ii)}} If $(\eps,u)\in \R_+^* \times \R_+^*$ is such that $\lambda(B(1,u,\eps))>0$, then,  \eqref{sys1} admits a unique locally stable equilibrium 
$E_{\eps,u}:=(x^{\eps,u},s^{\eps,u})\in (0,+\infty)^n\times (0,1)$ called coexistence steady-state and $E_{wo}$ is unstable. 
 \end{prop}
\begin{proof} For sake of completeness, we give the proof of (i) which can also be found in \cite{deleen}. If $(x,s)$ is a steady-state of \eqref{sys1}, then 
\begin{equation}{\label{eq-steady-state}}
\left\{
\begin{array}{ll}
B(s,u,\eps)x&=0,\vspace{0.1cm}\\
\sum_{j=1}^n \mu_j(s)x_j&=u(1-s). 
\end{array}
\right.
\end{equation}
The equation $B(s,u,\eps)x=0$ with $x\not=0$ implies that 
0 is an eigenvalue, hence $\lambda(B(s,u,\eps))\geq 0$. It is possible only if $s=1$. Indeed, otherwise, since $s<1 \; \Rightarrow \; \lambda(B(s,u,\eps))<\lambda(B(1,u,\eps))$, we would have $\lambda(B(s,u,\eps))<0$ and a contradiction. It follows from \eqref{eq-steady-state} 
that any equilibrium verifies $x=0$, so, the only equilibrium point is the washout. 
The Jacobian of \eqref{sys1} at $E_{wo}$ is the block matrix
$$
\left[
\begin{array}{cc}
B(1,u,\eps) & 0\\
-\mu_1(1)\cdots -\mu_n(1) & -u
\end{array}\right]\in \R^{(n+1)\times (n+1)}.
$$
If $\lambda(B(1,u,\eps))\leq 0$, then, 
$\dot{x}=B(s,u,\eps)x\leq B(1,u,\eps)x$, thus there is $C_1>0$ such that $\|x(t)\| \leq C_1 \|x(0)\|$ for every $t\geq 0$. Using 
that $\Delta$ is attractive for \eqref{sys1} (recall Lemma \ref{invaLem}), we deduce that $E_{wo}$ is stable. If, in addition, $\lambda(B(1,u,\eps))<0$,  $B(1,u,\eps)$ is a Hurwitz matrix, and, thanks to the inequality $\dot{x}\leq B(1,u,\eps)x$, we deduce that $x(t)\rightarrow 0$ as $t\rightarrow +\infty$ which proves the desired property using Lemma \ref{invaLem}. 

In case (ii), we find two equilibria depending if $s=1$ or $s<1$. If $s=1$, then $x=0$ and the corresponding steady-state is the washout that is unstable since $\lambda(B(1,u,\eps))>0$.  
The other possible steady-states satisfy \eqref{eq-steady-state} with $s<1$, so, $x\in \mathrm{Ker}(B(s,u,\eps))\backslash \{0\}$. 
But, the largest eigenvalue of $B(s,u,\eps)$ is the only one with a positive eigenvector (thanks to the Perron-Frobenius Theorem). So, we necessarily have $\lambda(B(s,u,\eps))=0$ which has a unique solution $s^{\eps,u}$ 
(because of the monotonicity of $\lambda(B(\cdot,u,\eps))$ and the fact that $\lambda(B(0,u,\eps))\lambda(B(1,u,\eps))<0$). Hence, zero is the Perron root of $B(s^{\eps,u},u,\eps)$ (it is a simple eigenvalue) and we denote by $a^{\eps,u}\in \R^n$ its Perron vector. We deduce that $x$ necessarily satisfies
$$
x=\nu a^{\eps,u} \quad \mathrm{and} \quad \sum_{j=1}^n x_j+s^{\eps,u}=1, 
$$
where $\nu\in \R_+^*$. These two equalities define a unique point $x^{\eps,u}\in \R^n$ such that $x^{\eps,u}_i>0$ for every $1 \leq i \leq n$. 
Let us now turn to the local asymptotic stability property. Observe that \eqref{sys1} is equivalent to
$$
\left|\begin{array}{cl}
\dot{x}&= B(b - \sum_{j=1}^n x_j,u,\eps)x,\vspace{0.1cm} \\
\dot{b}&=u(1-b),
\end{array}\right.
$$
(recall that $b=s+\sum_{j=1}^n x_j$). The Jacobian matrix of the preceding system at $(x^{\eps,u},1)$ is
$$
J_{\eps,u}:=\left[
\begin{array}{cc}
A_{\eps,u} & d_{\eps,u}\\
0 & -u
\end{array}\right]\in \R^{(n+1)\times (n+1)},
$$
where $d_{\eps,u}:=M'(s^{\eps,u})x^{\eps,u}\in \R^n$ and $A_{\eps,u}:=B(s^{\eps,u},u,\eps)- d_{\eps,u} a^\top\in \R^{n\times n}$ is a rank-one perturbation of $B(s^{\eps,u},u,\eps)$. 
For proving our claim, it is then 
enough to show that $A_{\eps,u}$ is a Hurwitz matrix. Observe that one has
$-B(s^{\eps,u},u,\eps)=\rho(H) I_n-H$ where $H:=\rho(B(s^{\eps,u},u,\eps)) I_n+B(s^{\eps,u},u,\eps)$. Hence, $-B(s^{\eps,u},u,\eps)$ can be written as a singular $M$-matrix which is thus non-negative.  
 We can now apply Theorem 2.7 (v) of \cite{Bierkens}  with the matrix $H$ (for this, note that $d_{\eps,u}$ and  $a$ have positive coefficients) and deduce that 
 $-A_{\eps,u}$ is strictly positive stable (which means that all eigenvalues of $-A_{\eps,u}$ are with positive real parts). We can thus conclude that 
  $J_{\eps,u}$ is a Hurwitz matrix which ends the proof. 
\end{proof}
In the rest of the paper, we keep the notation $a^{\eps,u}$ for the Perron vector associated with the $0$ eigenvalue of the matrix $B(s^{\eps,u},u,\eps)$. 
\subsection{Occurrence of the washout and coexistence steady-states}{\label{occu-sec}}
In this part, we make more explicit the condition about $\lambda(B(1,u,\eps))$ which separates  washout and  coexistence equilibria in Proposition \ref{disjonction}. It is convenient to introduce the functions 
$$
\bar \mu(s):=\frac{1}{n}\sum_{j=1}^n \mu_j(s) \; ; \; \hat \mu(s):=\max(\mu_1(s),...,\mu_n(s)),  
$$
which are increasing over $\R_+$. Also, we set $m:=\hat \mu(1)$. 
\begin{prop}{\label{prop-ineg}} For every $(\eps,u)\in \R_+^* \times \R_+^*$ such that $\lambda(B(1,u,\eps))>0$, one has:
\begin{equation}{\label{Seq}}
\bar \mu(s^{\eps,u})\leq u \leq \hat \mu(s^{\eps,u}) \leq u+2\eps.
\end{equation}
In addition, for every $(\eps,u)\in \R_+^* \times \R_+^*$, the quantity $\lambda(B(1,u,\eps))$ satisfies the following inequalities:
\begin{equation}{\label{Sin}}
\max\Big{(}\hat \mu(1)-u-2\eps,\bar \mu(1)-u\Big{)} \leq \lambda(B(1,u,\eps))\leq 
\hat \mu(1)-u.
\end{equation}
\end{prop}
\begin{proof} 
First, observe that $B(s,u,\eps)\leq \hat B(s)$ (componentwise) where $\hat B(s):=\max_{1 \leq j \leq n} (\mu_j(s)-u)I_n+\eps T\in \R^{n\times n}$, and that both matrices are quasi-positive and irreducible. We can thus deduce that $\lambda(B(s,u,\eps))\leq \lambda(\hat B(s))=\max_{1 \leq j \leq n}\mu_j(s)-u$. This gives the second inequality in \eqref{Sin}. 

Now, for $s\in [0,1]$, we set $\zeta:=\lambda(B(s,u,\eps))$ and let $x$ be an eigenvector of $B(s,u,\eps)$ with eigenvalue $\zeta$. Recall from Perron-Frobenius's Theorem that $x>0$. 
The equality $B(s,u,\eps)x=\zeta x$ rewrites
\begin{equation}{\label{ineg-tmp}}
\left\{
\begin{array}{ll}
(\mu_1(s)-u-\eps)x_1+\eps x_2&=\zeta x_1,\\
(\mu_i(s)-u- 2 \eps)x_i+\eps (x_{i+1}+x_{i-1})&=\zeta x_i, \quad 2 \leq i \leq n-1,\\
(\mu_n(s)-u-\eps)x_n+\eps x_{n-1}&=\zeta x_n,
\end{array}
\right.
\end{equation}
Summing those $n$ equalities with $s=s^{\eps,u}$ and $x=x^{\eps,u}$  gives 
$$
\sum_{j=1}^n \mu_j(s^{\eps,u})x_j^{\eps,u}=u \sum_{j=1}^n x_j^{\eps,u},
$$
thus, one obtains $u=\sum_{j=1}^n p_j^{\eps,u} \mu_j(s^{\eps,u})$ (where for $1 \leq i \leq n$, $p_i^{\eps,u}:=\frac{x_i^{\eps,u}}{\sum_{j=1}^n x_j^{\eps,u}}$). Because $p_i^{\eps,u} \leq 1$, we deduce that 
$u \leq \hat \mu(s^{\eps,u})$, which gives the second inequality in \eqref{Seq}.
From \eqref{ineg-tmp} with $s\in [0,1]$, we also get
\begin{equation}{\label{ineg-tmp2}}
\zeta=\mu_i(s)-u-a_i \eps+\eps \frac{x_{i+1}+x_{i-1}}{x_i},
\end{equation}
with $a_i=2$ for $2 \leq i \leq n-1$, $a_1=a_n=1$, and the convention that $x_0=x_{n+1}=0$. 
Since $x_i>0$ for $1 \leq i \leq n$, 
this equality entails $$\zeta \geq \max_{1 \leq j \leq n} \mu_j(s)-u-2\eps.$$ 
From the preceding inequality, we can deduce the last inequality in \eqref{Seq} (taking $s=s^{\eps,u})$ and also the inequality 
$\lambda(B(1,u,\eps))\geq  \hat \mu(1) -u-2\eps$ in \eqref{Sin} (taking $s=1$).  

To conclude, we need to prove the two  inequalities in \eqref{Seq}-\eqref{Sin}  involving the mean values of the kinetics. Doing so, we sum  equalities \eqref{ineg-tmp2} which gives
$$
\zeta=\frac{1}{n}\sum_{j=1}^n \mu_j(s)-u-\frac{2(n-1)}{n}\eps +\frac{\eps}{n}\left[\sum_{j=1}^{n-1}\frac{x_{j+1}}{x_j}+\sum_{j=2}^{n}\frac{x_{j-1}}{x_j}\right].
$$
Applying the arithmetic-geometric mean inequality yields
$$
\frac{1}{n-1}\sum_{j=1}^{n-1}\frac{x_{j+1}}{x_j} \geq \left(\prod_{j=1}^{n-1}\frac{x_{j+1}}{x_j}\right)^{\frac{1}{n-1}}= \left(\frac{x_n}{x_1}\right)^{\frac{1}{n-1}} \; ; \; 
\frac{1}{n-1}\sum_{j=1}^{n-1}\frac{x_{j-1}}{x_j} \geq \left(\frac{x_1}{x_n}\right)^{\frac{1}{n-1}},
$$

which implies
$$
\zeta \geq \frac{1}{n}\sum_{j=1}^n \mu_j(s)-u-\frac{2(n-1)}{n}\eps + \frac{(n-1)\eps}{n}\left[\left(\frac{x_n}{x_1}\right)^{\frac{1}{n-1}}+\left(\frac{x_1}{x_n}\right)^{\frac{1}{n-1}}\right]. 
$$
Using that $y^\frac{1}{n-1}+y^{-\frac{1}{n-1}} \geq 2$ for every $y>0$, we  obtain the inequality $\zeta \geq \bar \mu(s)-u$. 
Specializing this  inequality with $s=s^{\eps,u}$ and $s=1$ gives us the left  inequalities in \eqref{Seq} and \eqref{Sin} which concludes the proof. 
\end{proof}
%
Thanks to this proposition, we can make the following observations:
\begin{itemize}
\item[$\bullet$] From \eqref{Seq}-\eqref{Sin}, the steady-state $E_{\eps,u}$ occurs whenever $\bar \mu(1)> u$ and this condition does not depend on the parameter $\eps$. 
\item[$\bullet$] Observe also that, thanks to those inequalities, we recover the fact that if $u \geq m$, then, only washout occurs. 
\end{itemize}
For every $\eps\geq 0$, we can also uniquely define a critical value for the dilution rate  
$$
u_c(\eps):=\lambda(M(1)+\eps T), 
$$
which is such that only washout occurs if the dilution rate is such that $u\geq u_c(\eps)$ (according to Proposition \ref{disjonction} (i)). 
From \eqref{Sin} and the previous remarks, the value $u_c(\eps)$ satisfies:
\begin{equation}{\label{borne-uc}}
\forall \eps \geq 0, \; u_c(\eps)\in [\hat u(\eps),m],
\end{equation}
where $\hat u(\eps):=\max(m-2\eps,\bar \mu(1))$. Interestingly, the presence of mutation in the system implies occurrence of the washout for values of the dilution rate in the interval $[m-2\eps,m)$ 
for which the species with the least break-even concentration would survive (without mutation). 
In addition, we can observe that  the larger the mutation rate is, the lower the dilution need to be to avoid washout. 

We now recall a result related to the differentiability of $\lambda(\cdot)$ that will be applied several times in this paper. 
Given a symmetric quasi-positive matrix  $A$, the largest eigenvalue of $A$ is simple and thus $\lambda(\cdot)$ is analytic as a function of its $n^2$ coefficients in some neighborhood of $A$ in the symmetric matrices space (see, {\it{e.g.}}, \cite{YY2016,DN84}). 
In addition, the first derivative of $\lambda(\cdot)$ ({\it{i.e.}}, the matrix whose $(i,j)$ entry is $\frac{\partial \lambda}{\partial a_{i,j}}(A))$ is given by
\begin{equation}{\label{deriv-lambda}}
D_1\lambda(A)=w w^\top,
\end{equation}
where $w$ denotes the Perron vector associated with $\lambda(A)$ (see \cite{DN84,Harker87}).

\begin{prop}{\label{uc decreasing}}
The function $\eps \mapsto u_c(\eps)$ is non-increasing over $\R_+$. In addition, one has $u_c(0)=m$, and $u_c(\eps)\rightarrow \bar \mu(1)$ as $\eps\rightarrow +\infty$. 
\end{prop}
\begin{proof} 
Applying the previous property with the symmetric quasi-positive matrix 
$B(1, 0,\eps) = M(1)+\eps T$ gives
$$
u_c'(\eps) = \sum_{1 \leq i,j \leq n} \frac{\partial \lambda(B(1,0,\eps))}{\partial b_{i,j}}\frac{\partial b_{i,j}(1,0,\eps)}{\partial \eps} = (v^\eps)^\top T v^\eps \leq 0,
$$
where $v^\eps$ is the Perron vector associated with $B(1, 0,\eps)$ and $b_{i,j}(1,0,\eps)$ denote the $n^2$ entries of $B(1,0,\eps)$. 
This shows that $u_c(\cdot)$ is non-increasing over $\R_+$. 
Now, from the CEP, we have immediately $u_c(0)=m$. Finally, 
using \eqref{deriv-lambda}, we have the expansion
\begin{equation}{\label{uc limite}}
    u_c(\eps) = \eps \lambda \left( T + \frac{1}{\eps} M(1) \right) = \eps \left[ \lambda(T) + \frac{1}{\eps} \frac{a^\top M(1) a}{a^\top a} + o\left(\frac{1}{\eps}\right) \right] = \bar \mu(1) + o(1),
\end{equation}
as $\eps \rightarrow +\infty$, which concludes the proof. 
\end{proof}

\subsection{Global stability in the two species case}
In this part, we prove that $E_{\eps,u}$ is GAS for \eqref{sys1} when $n=2$ and we also give explicit expressions for the steady-state and the critical value of the dilution rate. 
We start by addressing the global stability property. 
\begin{prop} 
For $n=2$ and $(\eps,u) \in \R_+^* \times (0,u_c(\eps))$, $E_{\eps,u}$ is globally asymptotically stable in $\mathcal{D}$. 
\end{prop}
\begin{proof} 
For those initial conditions in the set $\Delta$, \eqref{sys1} is equivalent to
\begin{equation}{\label{sys-auto2}}
\dot{x}=h(x), 
\end{equation}
where $h:\mathcal{D}'\rightarrow \R^2$ is given 
by (recall \eqref{setF})
$$
h(x):=\left(\begin{array}{c}(\mu_1(1-x_1-x_2)-u)x_1+\eps(x_2-x_1)\\(\mu_2(1-x_1-x_2)-u)x_2+\eps(x_1-x_2)\end{array}\right).
$$
If we set $\vphi(x_1,x_2):=\frac{1}{x_1x_2}$ for $x_1,x_2>0$, a direct computation shows that the quantity
$$
\frac{\partial (\vphi h_1)}{\partial x_1}+\frac{\partial (\vphi h_2)}{\partial x_1}=-\eps\frac{x_1^2+x_2^2}{x_1^2x_2^2}-\frac{\mu'_1(1-x_1-x_2)}{x_2}-\frac{\mu'_2(1-x_1-x_2)}{x_2},
$$
is negative in the interior of $\mathcal{D}'$. It follows from the Bendixson-Dulac Theorem that no periodic orbits occurs 
in $\mathcal{D}'$ for \eqref{sys-auto2}. 
Now, Proposition \ref{disjonction} implies that only two equilibria occur, namely the washout $(0,0)$ which is unstable and the point $x^{\eps,u}$ in the interior of $\mathcal{D}'$ which is locally asymptotically stable. Since there are no periodic orbits, we deduce that $x^{\eps,u}$ is globally asymptotically stable for \eqref{sys1} restricted to $\Delta$. 
Now, coming back to \eqref{sys1} for $n=2$, the sub-system satisfied by $x$ reads 
\begin{equation}{\label{sys-nonauto2}}
\dot{x}=\tilde h(t,x),
\end{equation} 
where $\tilde h:\R \times \R^2 \rightarrow \R^2$ is defined as
$$
\tilde h(t,x):=\left(\begin{array}{c}(\mu_1(b(t)-x_1-x_2)-u)x_1+\eps(x_2-x_1)\\(\mu_2(b(t)-x_1-x_2)-u)x_2+\eps(x_1-x_2)\end{array}\right).
$$
Clearly, $\tilde h(t,x)\rightarrow h(t,x)$ uniformly locally w.r.t.~$x$, thus \eqref{sys-nonauto2} is a non-autonomous perturbation of \eqref{sys-auto2}. 
Since every solution to \eqref{sys-auto2} converges to $x^{\eps,u}$, 
we deduce from \cite{Thieme92} that every solution to \eqref{sys-nonauto2} also converges to this point. To conclude, let us given a solution $(x(\cdot),s(\cdot))$ of \eqref{sys1}. We have proved that 
$x(t)\rightarrow x^{\eps,u}$ when $t$ goes to infinity. Since $s(t)+x_1(t)+x_2(t)\rightarrow 1$ as $t\rightarrow +\infty$, we deduce that $s(t)\rightarrow s^{\eps,u}$ as $t$ goes to infinity which ends the proof.  
\end{proof}
We now turn to explicit expressions involving the steady-state. The sub-system $(x_1,x_2)$ of \eqref{sys1} rewrites
$$
\dot{x}=B(s,u,\eps)x \quad \mathrm{with} \quad 
B(s,u,\eps)=\left[
\begin{array}{cc}
\mu_1(s)-(u+\eps) & \eps\\
\eps & \mu_2(s)-(u+\eps)
\end{array}
\right].
$$
The largest eigenvalue of $B(s,u,\eps)$ can thus be explicitly computed:
\begin{equation}{\label{eq-s}}
\lambda(B(s,u,\eps))=-\eps-u+\frac{\mu_1(s)+\mu_2(s)}{2}+\frac{1}{2}\sqrt{(\mu_1(s)-\mu_2(s))^2+4\eps^2}.
\end{equation}
Hence, the coexistence steady-state $(x^{\eps,u},s^{\eps,u})$ exists provided that $\lambda(B(1,u,\eps))>0$ (see Proposition \ref{disjonction}) which amounts to saying that the dilution rate fulfills the inequality
$$
u<u_c(1)=-\eps+\frac{\mu_1(1)+\mu_2(1)}{2}+\frac{1}{2}\sqrt{(\mu_1(1)-\mu_2(1))^2+4\eps^2}. 
$$
For a given $u$ satisfying the previous inequality, we can compute $s^{\eps,u}$ numerically solving $\lambda(B(s,u,\eps))=0$ w.r.t.~$s$, thanks to \eqref{eq-s} (see Fig.~\ref{fig-2esp}). 
It follows that the dilution rate $u$ is related to $s^{\eps,u}$ via the equality
$$
u=-\eps+\frac{\mu_1(s^{\eps,u})+\mu_2(s^{\eps,u})}{2}+\frac{1}{2}\sqrt{(\mu_1(s^{\eps,u})-\mu_2(s^{\eps,u}))^2+4\eps^2}.
$$

Using that $x^{\eps,u}_1+x^{\eps,u}_2+s^{\eps,u}=1$, one also obtains
$$
x^{\eps,u}_1=\frac{(1-s^{\eps,u})(u-\mu_2(s^{\eps,u}))}{\mu_1(s^{\eps,u})-\mu_2(s^{\eps,u})} \; ; \; x^{\eps,u}_2=\frac{(1-s^{\eps,u})(u-\mu_1(s^{\eps,u}))}{\mu_2(s^{\eps,u})-\mu_1(s^{\eps,u})}. 
$$
\begin{rem}
The previous expressions of $x^{\eps,u}_i$ are valid if the kinetics do not intersect. If there is a (unique) $\bar s\in (0,1)$ 
such that $\mu_1(\bar s)=\mu_2(\bar s)$, then, these expressions are valid only if $u\not= \mu_1(\bar s)$. If $u=\bar u:=\mu_1(\bar s)$, then, one has $s^{\eps,\bar u}=\bar s$ and 
$x^{\eps,\bar u}=\big{(}\frac{1-\bar s}{2},\frac{1-\bar s}{2}\big{)}$. 
\end{rem}
Fig.~\ref{fig-2esp} depicts $\eps \mapsto x^{\eps,u}$, $\eps \mapsto s^{\eps,u}$, and $\eps \mapsto u_c(\eps)$ for a fixed $u>0$ such that species $2$ survives when $\eps=0$ 
(see the plof of $\mu_1$ and $\mu_2$ below). We verify numerically that $E_{\eps,u}\rightarrow E_2$  as $\eps \downarrow 0$ (see Section \ref{sec-dev-lim}) and that 
$u_c(0)=\max(\mu_1(1),\mu_2(1))$ and $u_c(+\infty)=\frac{\mu_1(1)+\mu_2(1)}{2}$ (see \eqref{uc limite}). 
\begin{figure}[!ht]
\centering
\includegraphics[width=3.2cm]{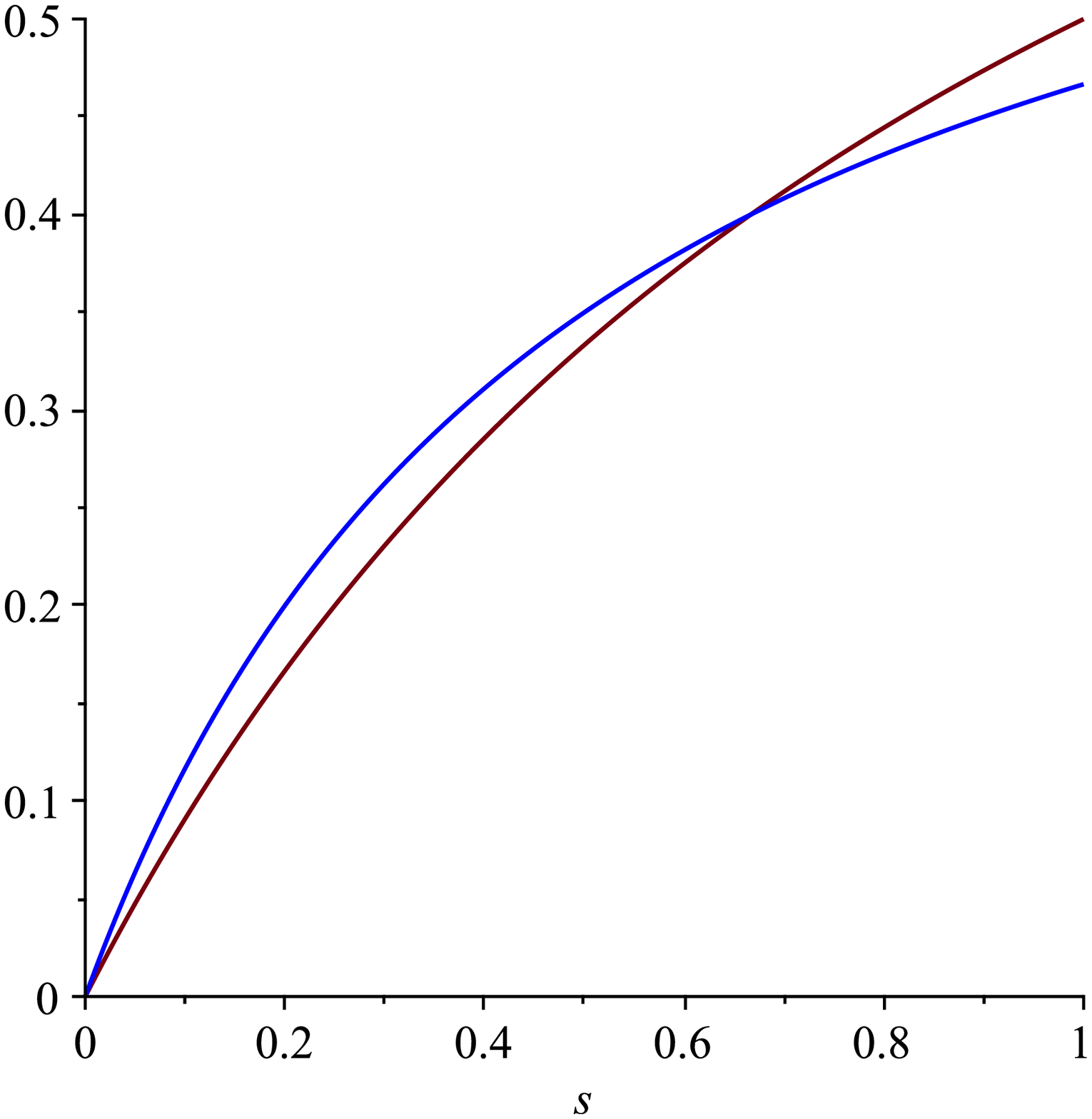}
\includegraphics[width=3.2cm]{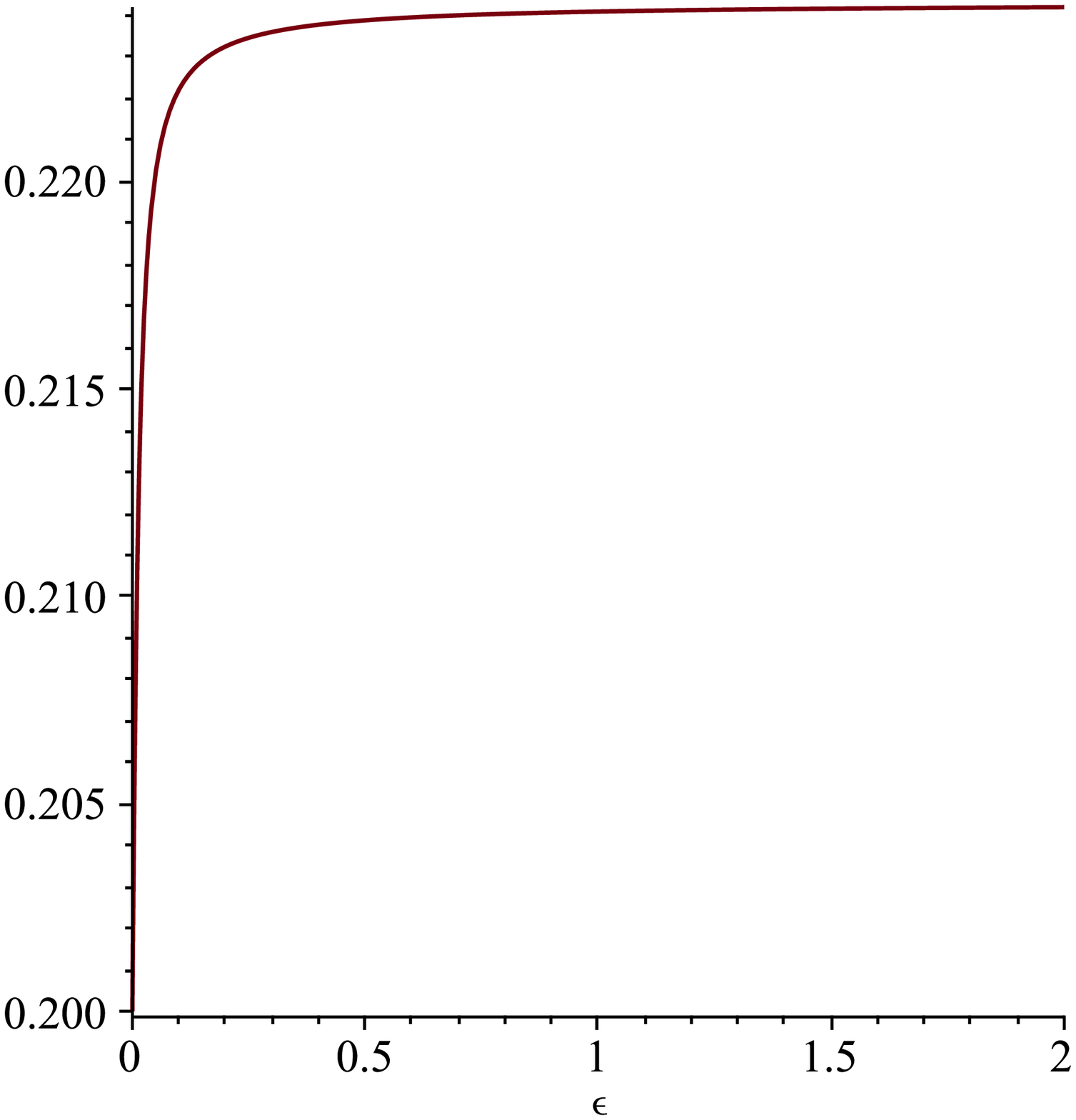}
\includegraphics[width=3.2cm]{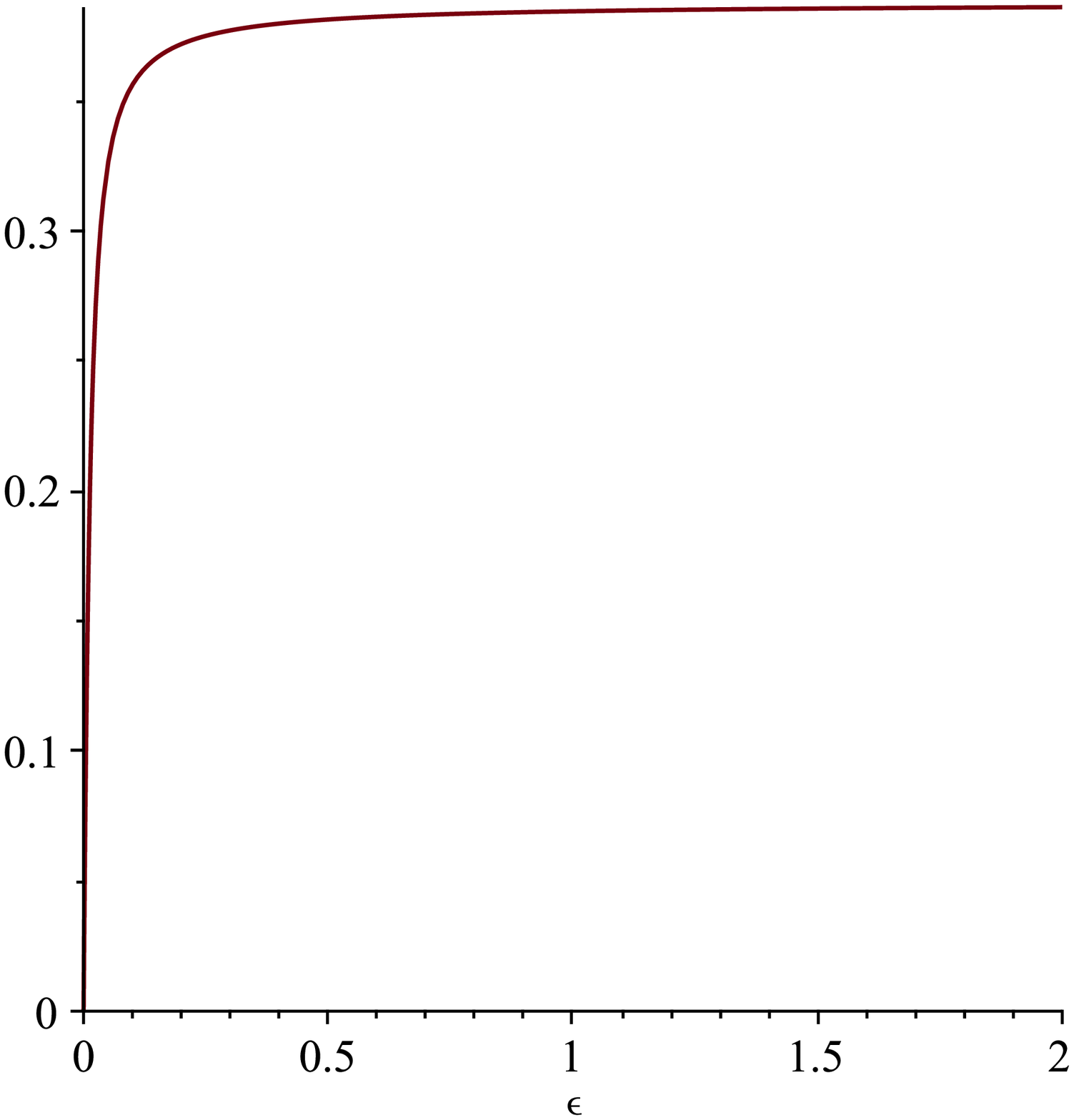}
\includegraphics[width=3.2cm]{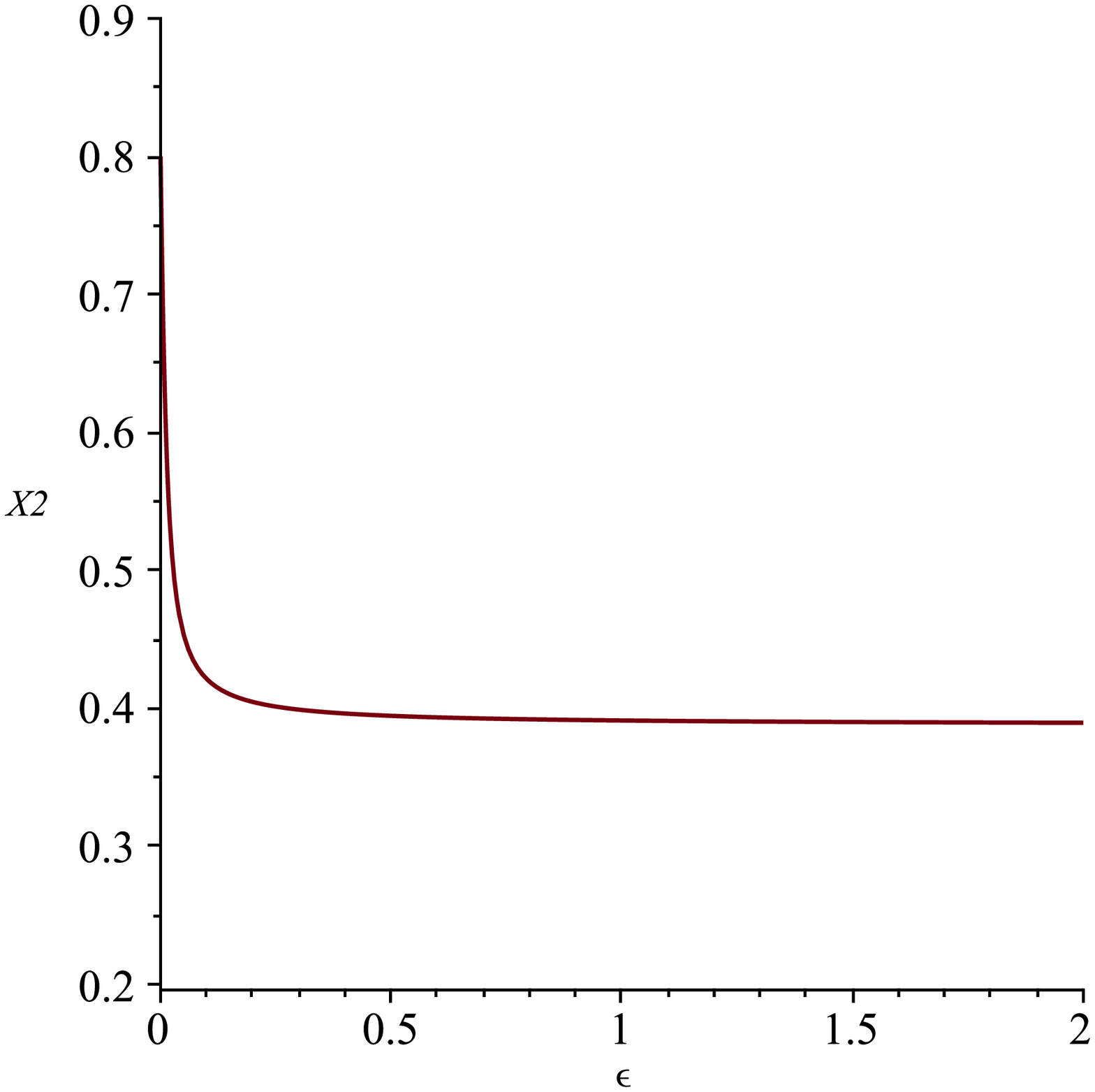}
\includegraphics[width=3.2cm]{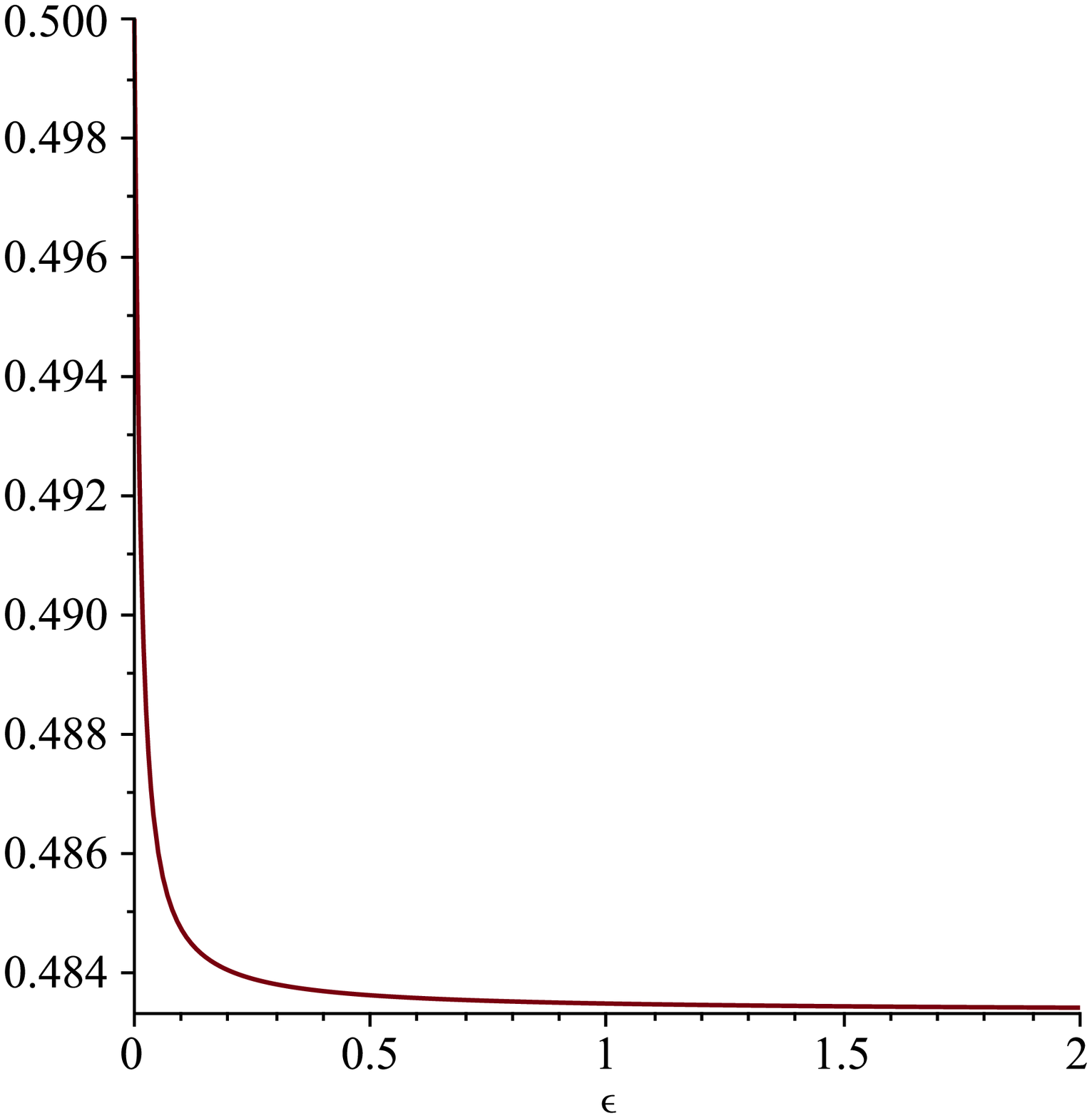}
\caption{Plot of $\mu_1(s):=\frac{s}{s+1}$ (in red) and $\mu_2(s):=\frac{0.7s}{0.5+s}$ (in blue) on Fig.~left. For $u=0.2$ and $\eps=0$, species $2$ survives and $\lambda_2(0.2)=0.2$ (recall \eqref{bec}).  Next, from left to right, plot of $s^{\eps,u}$, $x^{\eps,u}_1$, and $x^{\eps,u}_2$ as a function of $\eps\in [0,2]$. Fig.~right depics $u_c(\eps)$ illustrating \eqref{uc limite}.} 
\label{fig-2esp}
\end{figure}
\subsection{Illustration of the global stability property for $n\geq 3$}
In view of the local stability property of $E_{\eps,u}$ and the global stability of this equilibrium for $n=2$, one can wonder if this property 
remains valid for $n \geq 3$, $\eps>0$, and $u\in (0,u_c(\eps))$.  
Although we know the behavior of \eqref{sys1} for $\eps=0$, it turns out that this question is delicate even if $\eps$ is arbitrarily small (see also Remark \ref{lyapu-rem}).  
In Section \ref{sec-4}, we address this question when $\eps>0$ is fixed and $u$ 
is with small enough values.  

We present below numerical simulations of solutions to \eqref{sys1} for $n=20$, $\eps=1$, and $u=5$, see Fig.~\ref{fig-simu-global}. The kinetics associated with the species are arbitrary functions of Monod type.  
Our observations are as follows:
\begin{itemize}
\item[$\bullet$] First, we  
observe convergence of the system to the 
coexistence equilibrium for a large set of initial conditions. 
\item[$\bullet$] Interestingly, we also see that even though the system converges to the coexistence equilibrium, 
very few species have a significant concentration asymptotically. We shall give an explanation of this phenomenon in  Section \ref{sec-dev-lim} 
for small values of the parameter $\eps>0$.  
\end{itemize}
\begin{figure}[h!]
\centering
\includegraphics[width=5cm]{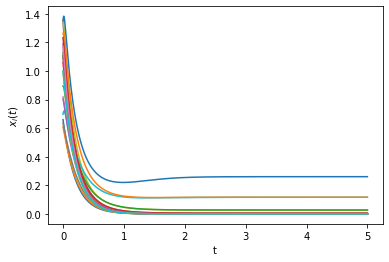}
\includegraphics[width=5cm]{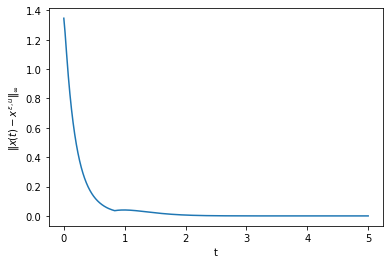}
\includegraphics[width=5cm]{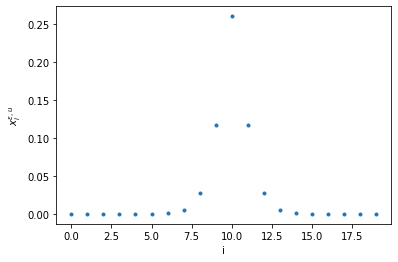}
\caption{{\it{Fig.~left}} : plot of $x_i(\cdot)$. {\it{Fig.~middle}}: plot of the $L^\infty$-norm error $\|x(\cdot)-x^{\eps,u}\|_{L^\infty}$ between the solution and the steady-state. {\it{Fig. right}}: 
plot of the value of species concentrations when $t$ goes to infinity. Data are such that $n = 20$, $\eps=1$, $u = 5$ 
and $\mu_i(s)=\frac{20 s}{a_i+s}$ with $a_i = 1+ \frac{1}{2}(10-i)^2$, $1 \leq i \leq 20$.} 
\label{fig-simu-global}
\end{figure}

\section{Behavior of the coexistence steady state}{\label{sec-dev-lim}}
In this part, we study the behavior of the coexistence equilibrium w.r.t.~the parameter $\eps$.  Based on the implicit function theorem, we give an expansion of $E_{\eps,u}$ up to the first order as $\eps \downarrow 0$ (for a fixed dilution rate $u$) and we also study its limit as  $\eps\rightarrow +\infty$. Recall that $m=\hat \mu(1)$. 
\begin{prop}{\label{prop-expan}}
Suppose that $u \in (0,m)$ and that there is a unique $1 \leq i_0 \leq n$ such that $\lambda_{i_0}(u)=\min_{1 \leq i \leq n} \lambda_i(u)<+\infty$. 
Then, there exist $\xi\in \R^{n}$ and $\sigma>0$ such that when $\eps \downarrow 0$, the following expansion is fulfilled:
\begin{equation}{\label{DL-s}}
E_{\eps,u}=E_{i_0}+\eps (\xi,\sigma)+o(\eps). 
\end{equation}
In addition, the vector $\xi$ and $\sigma$ are given by (with the convention that $\xi_0=\xi_{n+1}=0$):
\begin{equation}{\label{expr-xi}}
\left\{\begin{array}{cll}
\xi_j&=0 & \mathrm{if} \; j\in \{1,...,n\} \backslash \{i_0-1,i_0,i_0+1\}, \vspace{0.1cm}\\
\xi_{j}&=\frac{1-\lambda_{i_0}(u)}{u-\mu_j(\lambda_{i_0}(u))} & \mathrm{if} \; j\in \{i_0-1,i_0+1\} \backslash \{0,n+1\},\vspace{0.1cm}\\
\xi_{i_0}&=-\xi_{i_0-1}-\xi_{i_0+1}-\sigma, \vspace{0.1cm}& \\
\sigma &=-\frac{t_{i_0,i_0}}{\mu'_{i_0}(\lambda_{i_0}(u))}.
\end{array}\right.
\end{equation}
\end{prop}
\begin{proof} Since $0<u<m$, one has $m-2\eps>u$ for $\eps$ small enough, thus, \eqref{borne-uc} implies that $u<\hat u(\eps)\leq u_c(\eps)$ so that the steady-state $E_{\eps,u}$ exists for every $\eps>0$ small enough.

Now, for convenience, we write $E_{i_0}$ as $E_{i_0}=(x^*,s^*)$. 
We start by proving that the mapping $\eps\mapsto s^{\eps,u}$ is of class $C^1$ in some right neighborhood of $\eps=0$. Doing so, let $\theta>0$ and let us define the open set 
$D_\theta:=(-\theta,1)\times (0,1)$. Consider  the $C^1$ mapping $\tilde B:D_\theta\rightarrow \R^{n\times n}$ given by
$\tilde B(\eps,s):=B(s,u,\eps)$ for $(\eps,s)\in D_\theta$ (here $u=\mu_{i_0}(s^*)>0$ is fixed). 
Note that for every $(\eps,s)\in D_\theta$, the matrix $\tilde B(\eps,s)$ is symmetric quasi-positive and that for $\eps=0$, zero is the largest and simple eigenvalue of $\tilde B(0,s^*)=M(s^*)-u I_n$  
(observe that $\tilde B(0,s^*)$ is diagonal with $n$ distinct eigenvalues). It follows that $\lambda(\cdot)$ is analytic as a function of its $n^2$ coefficients in some neighborhood  
of $\tilde B(0,s^*)$ in the space of symmetric matrices.  
Since $\tilde B(\cdot,\cdot)$ is of class $C^1$ w.r.t.~$(\eps,s)$, 
there are $\theta'>0$ and $\nu>0$ small enough such that the composition  
$$
(\eps,s)\mapsto \vphi(\eps,s):=\lambda(\tilde B(\eps,s))
$$
is of class $C^1$ over $(-\theta',\theta')\times (s^*-\nu,s^*+\nu)$. 
For $\eps=0$, the unitary eigenvector of $\tilde B(0,s^*)$ for the zero eigenvalue is the $i_0$-th vector $w=e_{i_0}$ of the canonical basis of $\R^n$.  
It follows from \eqref{deriv-lambda} that 
$\frac{\partial \vphi}{\partial s}(0,s^*)=\mu'_{i_0}(s^*)>0$. 
So, we can apply the implicit function theorem locally around $(0,s^*)$. Hence, the mapping $\eps \mapsto s^{\eps,u}$ is of class $C^1$ over $(-\eps_0,\eps_0)$, and in particular 
in some right neighborhood 
of $\eps=0$. In addition, one has:
$$
\forall \eps \in (-\eps_0,\eps_0), \quad \vphi(\eps,s^{\eps,u})=0.
$$
By differentiating the preceding equality w.r.t.~$\eps$ and letting $\eps\downarrow 0$, we find
\begin{equation}{\label{deriv-lambda-tmp}}
\sum_{1 \leq i,j \leq n} \frac{\partial \lambda(\tilde B(0,s^*))}{\partial \tilde b_{i,j}}\frac{\partial \tilde b_{i,j}(0,s^*)}{\partial \eps}+\mu'_{i_0}(s^*)\frac{d s^{\eps,u}}{d \eps}_{|_{\eps=0}}=0,
\end{equation}
where $\tilde b_{i,j}(0,s^*)$ denote the $n^2$ entries of $\tilde B(0,s^*)$. 
Combining \eqref{deriv-lambda} and \eqref{deriv-lambda-tmp}, we obtain  $$t_{i_0,i_0}+\mu'_{i_0}(s^*)\frac{d s^{\eps,u}}{d \eps}_{|_{\eps=0}}=0,$$ which implies 
$\frac{d s^{\eps,u}}{d \eps}_{|_{\eps=0}}=\sigma$ and the desired expansion of $s^{\eps,u}$ up to the first order as in \eqref{DL-s}-\eqref{expr-xi}. 

Let us now turn to the expansion of $x^{\eps,u}$ w.r.t.~$\eps$. Doing so, let us consider the $C^1$ mapping 
$\psi:(-\theta,\theta)\times \R^n \rightarrow \R^{n+1}$ defined as
$$
\psi(\eps,x):=\Big{(}\tilde B(\eps,s^{\eps,u})x,\sum_{j=1}^n x_j -1\Big{)}, 
$$
whose differential w.r.t.~$x$ at $(0,x^*)$ satisfies
$$
D_x\psi(0,x^*)h=\Big{(}\tilde B(0,s^*)h,\sum_{j=1}^n h_j\Big{)}. 
$$
We can check that the kernel of $D_x\psi(0,x^*)$ is reduced to $\{0\}$, so, $D_x\psi(0,x^*)$ is invertible. Hence, by the implicit function theorem, we can conclude that $\eps\mapsto x^{\eps,u}$ is of class $C^1$ in some neighborhood of $\eps=0$. To obtain the desired expansion of $x^{\eps,u}$, let us write $x^{\eps,u}=x^*+\eps d+o(\eps)$. Since 
$\sum_{j=1}^n x^{\eps,u}_j+s^{\eps,u}=1$, one has 
\begin{equation}{\label{expr-xeps-tmp}}
\sigma+\sum_{j=1}^n d_j=0. 
\end{equation}
Expanding $B(s^{\eps,u},u,\eps)$ w.r.t.~$\eps$ up to the first order, we get: 
\begin{align*}
0=B(s^{\eps,u},u,\eps)x^{\eps,u}&=B(s^*+\sigma\eps+o(\eps),u,\eps)x^{\eps,u}\\
&=(M(s^*+\sigma\eps+o(\eps))-u I_n+\eps T)x^{\eps,u}\\
&=(M(s^*)-u I_n+\eps T + \sigma M'(s^*) \eps + o(\eps))(x^*+\eps d + o(\eps))\\
&=\big{(}T x^*+\sigma M'(s^*)x^*+(M(s^*)-u I_n)d\big{)}\eps+o(\eps),
\end{align*}
using the relation $(M(s^*)-u I_n)x^*=0$ in the last equality. 
%
Hence, we deduce that 
$$
T x^*+\sigma M'(s^*)x^*+(M(s^*)-u I_n)d=0,
$$
which gives
$$
(Tx^*)_j+\sigma \mu'_j(s^*)x^*+(\mu_j(s^*)-u)d_j, \quad 1 \leq j \leq n. 
$$
In the case where $1<i_0<n$, we obtain \eqref{expr-xi} for $j\in \{1,...,n\}\backslash \{i_0\}$ from the preceding equation. For $j=i_0$, \eqref{expr-xeps-tmp} gives \eqref{expr-xi}. 
A similar computation gives \eqref{expr-xi} whenever $i_0=1$ or $i_0=n$, which concludes the proof. %
\end{proof}
%
From Proposition \ref{prop-expan}, species of index $i_0$ is the only one with a positive value at the zero order. Observe that it satisfies the inequality $\xi_{i_0}<\lambda_{i_0}(u)$. In addition,  
only neighbors of $i_0$ ({\it{i.e.}}, species with index $i_0-1$ or $i_0+1$) are significant up to the first order. 
Species with index $j\notin \{i_0-1,i_0,i_0+1\}$ are (asymptotically) not significant w.r.t.~species with index $i_0-1$ and $i_0+1$. 
We now turn to the case where $\eps$ tends to $+\infty$.
\begin{prop}
For every $u\in (0,\bar \mu(1))$, the point $E_{\eps,u}$ has a limit when $\eps\rightarrow +\infty$ and 
\begin{equation}{\label{limit-infinity}}
\lim_{\eps \rightarrow +\infty} E_{\eps,u} = \Big{(}\frac{1-\bar \mu^{-1}(u)}{n}a,\bar \mu^{-1}(u)\Big{)}.
\end{equation}
\end{prop}
\begin{proof}
Using the implicit function theorem locally around each $\eps>0$, 
we deduce that the derivative of $\eps \mapsto s^{\eps,u}$ w.r.t.~$\eps$ exists and is non-negative (see the proof of Proposition \ref{uc decreasing}). Hence, $s^{\eps,u}$ is non-increasing, and thus it admits a limit $s^u$ as $\eps$ goes to infinity because $s^{\eps,u}\in [0,1]$, for every $\eps>0$. 
By definition of $s^{\eps,u}$ we have:
\begin{align*}
    0 &= \lambda \left( B(s^{\eps,u},u,\eps) \right) 
      = \eps \lambda \left(T + \frac{M(s^{\eps,u})}{\eps}  \right) - u \\
      &= \eps \left[ \lambda \left(T + \frac{M(s^{\eps,u})}{\eps}  \right) - \lambda \left(T + \frac{M(s^u)}{\eps}  \right) \right] + \eps \lambda \left(T + \frac{M(s^u)}{\eps}  \right) - u\\
      & = \eps \left[ \lambda \left(T + \frac{M(s^{\eps,u})}{\eps}  \right) - \lambda \left(T + \frac{M(s^u)}{\eps}  \right) \right] + \frac{1}{n} \sum\limits_{j=1}^n \mu_j(s^u) - u + o(1),
\end{align*}
using the same expansion as in \eqref{uc limite}. 
As $\lambda(\cdot)$ is of class $C^1$ in some neighborhood of $T$, it is in particular 
locally Lipschitz, so, the first term goes to 0 as $\eps$ goes to infinity (using that $M(\cdot)$ is also of class $C^1$ and that $s^{\eps,u} \rightarrow s^u$ as $\eps \rightarrow +\infty$).  
Hence, one must have  $\bar \mu(s^u)=u$, that is $s^u=\bar \mu^{-1}(u)$.  
Let us now turn to the limit of $x^{\eps,u}$ as $\eps\rightarrow +\infty$. 
From the proof of Proposition \ref{disjonction}, the vector $x^{\eps,u}$ satisfies the system
\begin{equation}{\label{tmp-limit-infinity}}
M(s^{\eps,u})x^{\eps,u}-ux^{\eps,u}+\eps T x^{\eps,u} =0,
\end{equation}
and it is proportional to $a^{\eps,u}$:
\begin{equation}{\label{limit-tmp1}}
x^{\eps,u} = \frac{1-s^{\eps,u}}{\sum_{j=1}^n a^{\eps,u}_j}a^{\eps,u}.
\end{equation}
Since for every $\eps>0$, $\|a^{\eps,u}\|=1$, there is $\tilde a^u\in \R^n$ 
with $\|\tilde a^u\|=1$  such that, up to a sub-sequence, one has $a^{\eps,u} \rightarrow \tilde a^u$.  By passing to the limit in \eqref{tmp-limit-infinity}, we find that $T \tilde a^u=0$, thus $\tilde a^u=\frac{a}{\sqrt{n}}$. Now, $\tilde a^u$ is also the limit of every converging sub-sequence of $(a^{\eps,u})_\eps$, hence $\tilde a^u$ is the limit of $(a^{\eps,u})_\eps$.  
Letting $\eps\rightarrow +\infty$ in \eqref{limit-tmp1} then gives \eqref{limit-infinity}, which ends the proof. 
\end{proof}
Even if the case $\eps\rightarrow +\infty$ may have no meaning from an application point of view, this result shows that 
species are asymptotically uniformly distributed.
\section{Persistence of all the species}{\label{sec-persist}}
In this section, we give an extension of \cite[Theorem 3]{deleen}  showing that each species (individually) is persistent. We refer to \cite{ST11,Thieme93,Thieme99} for the mathematical theory of persistence. 
 The persistence result in \cite{deleen} is related to the total biomass ({\it{i.e.}}, the sum of the concentrations of the species). In our setting, it can be stated as follows. 
\begin{thm}[\cite{deleen}]{\label{deleen-prop}} 
There is $c>0$ such that for every $(\eps,u)\in \R_+^* \times (0,u_c(\eps))$ and 
every initial condition in the set $\mathcal{D}$, the unique solution of \eqref{sys1} associated with this initial condition satisfies  
\begin{equation}{\label{borneinfdeleen}}
\liminf_{t\rightarrow +\infty} \; \sum_{j=1}^n x_j(t)\geq \beta_{\eps,u}:=c  \frac{\min_{1 \leq j \leq n} v_j^{\eps,u}}{\max_{1 \leq j \leq n} v_j^{\eps,u}},  
\end{equation}
where $v^{\eps,u}$ is the Perron vector associated with the matrix $B(1,u,\eps)$.
\end{thm}
\begin{rem} It follows that for every $(\eps,u)\in \R_+^* \times (0,u_c(\eps))$,  one has (recall \eqref{repulsif}):
\begin{equation}{\label{limsups}}
c_u \leq \liminf_{t \rightarrow +\infty} s(t) \leq \limsup_{t \rightarrow +\infty} s(t) \leq 1-\beta_{\eps,u}<1,
\end{equation}
for every solution to \eqref{sys1} starting in $\mathcal{D}$. Section \ref{convergence0} investigates the particular case where $\eps>0$ and $u=0$. 
\end{rem}
Before proving that each species is uniformly persistent, let us recall some definitions of \cite{Buttler,Freedman} about the notion of persistence. 
Hereafter, the interior, resp.~the boundary of a set $A\subset \R^n$ is denoted by $\mathrm{Int}(A)$, resp.~$\partial A$, $B(x,r)$ denotes 
the open ball of center $x\in \R^n$ and radius $r>0$. Finally, for every $r>0$, we define
$S(A,r):=\{x\in \R^n \; ; \; d(x,A)\leq r\}$ where $d$ is a distance over $\R^n$ and $d(x,A):=\inf_{a\in A} d(x,a)$. 
Consider now a differential equation $\dot{x}=f(x)$ where $f:\R^n \rightarrow \R^n$ is smooth and  such that every solution to this equation is global. Let us  denote by $\mathcal{F}$ the associated flow.
\begin{defi}{\label{defi-omega}}
Given two non-empty subsets $\mathcal{Y},\mathcal{Z}\subset \R^n$, the sets $W^\pm(\mathcal{Y})$ stand respectively for 
$$
W^+(\mathcal{Y}):=\{x \in \mathcal{Z} \; ; \; \omega(x)\subset \mathcal{Y}\} \; ; \; W^-(\mathcal{Y}):=\{x \in \mathcal{Z} \; ; \; \alpha(x)\subset \mathcal{Y}\}, 
$$
where $\omega(x)$ and $\alpha(x)$ denote respectively the $\omega$-limit and $\alpha$-limit sets of some point $x\in \R^n$ for the flow $\mathcal{F}$. 
\end{defi}
\begin{defi}{\label{defi-persis}} 
Let  $E$ be a non-empty closed subset of $\R^n$ 
that is forward invariant by $\mathcal{F}$. We say that $\mathcal{F}$ is uniformly persistent related to $E$ 
if there is $\kappa>0$ such that for every initial condition in $\mathrm{Int}(E)$, the corresponding solution $x(\cdot)$ satisfies
\begin{equation}{\label{persis-henri}}
\liminf_{t\rightarrow +\infty} d(x(t),\partial E)>\kappa, 
\end{equation}
where $d$ is a distance over $\R^n$
\end{defi}
In the next theorem, we show that each species is uniformly persistent. The proof is based on \eqref{borneinfdeleen}. 
\begin{thm}{\label{thm-main-persist}}
For every $(\eps,u)\in \R_+^* \times (0,u_c(\eps))$, there exists $\gamma_{\eps,u}>0$ such that for every initial condition in the set $\mathcal{D}$, the unique 
solution of \eqref{sys1} associated with this initial condition satisfies 
\begin{equation}{\label{ineg-final2}}
\liminf_{t\rightarrow +\infty}x_i(t)\geq \gamma_{\eps,u}, 
\end{equation}
for every $1 \leq i \leq n$. 
\end{thm}
\begin{proof} 
Fix $\eps>0$, $u\in (0,u_c(\eps))$, and consider the sets $E_\eta$ given by
$$
E_\eta:=\Big{\{}(x,s)\in \R_+^n \times [0,1] \; ; \;  s+\sum_{j=1}^n x_j \leq 1+\eta\Big{\}},
$$
where $\eta>0$. Obviously, $E_\eta$ is a closed subset of $\R_+^n \times [0,1]$ that is  positively invariant by \eqref{sys1}. In addition, it is easily seen that its boundary satisfies
$$
\partial E_\eta=\Big{\{}
(x,s)\in \R_+^n \times [0,1] \; ; \; \exists \, i\in\{1,...,n\}, \; x_i =0 \; \;  \mathrm{or} \; \; s=0 \; \;  \mathrm{or}\; \; \sum_{j=1}^n x_j+s=1+\eta \Big{\}}.
$$
Let $(x^0,s^0)\in \mathcal{D}$ be an initial condition and let us denote by $(x(\cdot),s(\cdot))$ the corresponding solution of \eqref{sys1}. From Property \ref{propri-pos}, one has  $x_i(t)>0$ for every $t>0$ and $1 \leq i \leq n$. In addition, $b(t)\rightarrow 1$ as $t\rightarrow +\infty$ and $s(\cdot)$ cannot approach $0$ because of \eqref{limsups}. We deduce in particular that  
$S(\partial E_\eta,1)\cap \mathrm{Int}(E_\eta)$ is point dissipative (see \cite{Buttler,Freedman}), which means the following:
$$
\forall (x^0,s^0)\in S(\partial E_\eta,1)\cap \mathrm{Int}(E_\eta), \; \forall t>0, \;  (x(t),s(t))\in 
\mathrm{Int}(E_\eta).
$$
We now show that the maximal invariant subset $N$ of $\partial E_\eta$ by \eqref{sys1} is 
acyclic\footnote{This property amounts to verify that $N^c\cap W^{-}(N)\cap W^+(N)=\varnothing$ where $N^c$ is the complement of $N$ in $E_\eta$, see \cite{Buttler,Freedman} or \cite{Hirsch00,Thieme99} for a more detailed definition.} and isolated. First, observe that 
$N=\{0_{\R^n}\} \times [0,1+\eta]$ using  Property \ref{propri-pos}. Considering now the distance $d$ over $\R_+\times [0,1]$ defined as
$$
d((x,s),(x',s')):=\sum_{j=1}^n |x_j-x'_j|+|s-s'|,
$$
for $(x,s),(x',s')\in \R_+\times [0,1]$, one has using \eqref{borneinfdeleen}
$$
\liminf_{t\rightarrow +\infty} d((x(t),s(t)),N)=
\liminf_{t\rightarrow +\infty} \sum_{j=1}^n x_j(t)\geq \beta_{\eps,u}>0, 
$$
for every solution of \eqref{sys1} starting in $\mathcal{D}$. 
If now $N$ and $\partial E_\eta$ stand respectively for $\mathcal{Y}$ and $\mathcal{Z}$ in Definition \ref{defi-omega}, the previous inequality implies that $W^+(N)=N$. Hence $N$ is necessarily acyclic.  

Finally, the set $N$ is isolated because for every initial condition $(x^0,s^0)\in \mathcal{D}\backslash N$, \eqref{deleen-prop} implies the  
existence of $t_0\geq 0$ such that 
$$
\forall t\geq t_0, \; d((x(t),s(t)),N)= \sum_{j=1}^n x_j(t)\geq \frac{\beta_{\eps,u}}{2}>0. 
$$
We are now in a position to use \cite[Theorem 4.3]{Freedman} which asserts that the flow defined by \eqref{sys1} is uniformly persistent related to the set $E_\eta$ provided that there is $\delta>0$ such that 
\begin{equation}{\label{tmp-uniform-persis}}
W^+(N)\cap S(\partial E_\eta,\delta)\cap  \mathrm{Int}(E_\eta) =\varnothing. 
\end{equation}
But, \eqref{tmp-uniform-persis} is clearly verified with $\delta:=1$ because $W^+(N)\cap S(\partial E_\eta,\delta) \subset N \subset \partial E_\eta$, so we 
have proved that for every $\eta>0$, the flow defined by \eqref{sys1} is uniformly persistent related to the set $E_\eta$. To conclude the proof, fix $\eta>0$ and apply \eqref{persis-henri} with $E_\eta$ in place of $E$. 
Note that Property \ref{propri-pos} and Lemma \ref{invaLem} imply that every solution is necessarily with values in $\mathrm{Int}(E_\eta)$ over $\R_+^*$. 
Hence, we deduce that there exists $\kappa>0$ such that 
$$
\liminf_{t\rightarrow +\infty} d((x(t),s(t)),\partial E_\eta)\geq \kappa,
$$
for every solution starting in $\mathcal{D}$.  
In view of the definition of $\partial E_\eta$, we can write
$\partial E_\eta=\bigcup_{i=1}^n F_i \cup \tilde F$ where $F_i:=\{(x,s)\in \R_+^n \times [0,1] \; ; \; x_i=0\}$ and $\tilde F$ is the complement of $\bigcup_{i=1}^n F_i$ in $E_\eta$. It follows that for every $1 \leq i \leq n$ and every initial condition in $\mathrm{Int}(E_\eta)$, one has
\begin{equation}{\label{ineg-final}}
\kappa \leq 
\liminf_{t\rightarrow +\infty} d((x(t),s(t)),\partial E_\eta)
\leq  \liminf_{t\rightarrow +\infty} d((x(t),s(t)),F_i)=\liminf_{t\rightarrow +\infty} x_i(t).
\end{equation}
Finally, for every initial condition $(x^0,s^0)\in \mathcal{D}$, there is a time $t'_0\geq 0$ such that for every time $t\geq t'_0$, the associated solution to \eqref{sys1} satisfies $(x(t),s(t))\in \mathrm{Int}(E_\eta)$. Combining this property with \eqref{ineg-final} then yields the desired property \eqref{ineg-final2} with $\gamma_{\eps,u}:=\kappa$. 
\end{proof}
%
\section{Global stability property of \eqref{sys1}}{\label{sec-4}}
\subsection{Asymptotic behavior of \eqref{sys1} with $u=0$}{\label{convergence0}}
We start by studying \eqref{sys1} in batch mode, {\it{i.e.}}, we take $u=0$. This will be useful to prove Theorem \ref{GAS-thm}. 
The dynamics of $x$ then becomes
\begin{equation}{\label{tmp1}}
\dot{x}=M(s(t))x+\eps Tx. 
\end{equation}
If $\eps=0$, the solution $(x(\cdot),s(\cdot))$ 
of \eqref{chem1} converges to some point $(x^\infty,0)\in \Delta$ such that $\sum_{j=1}^n x^\infty_j=b(0)$. So, we suppose in what follows that $\eps>0$. 
\begin{prop}{\label{dilutionratezero}}
If $u=0$ and $\eps>0$, every solution $(x(\cdot),s(\cdot))$ of \eqref{sys1} starting in $\mathcal{D}$ satisfies 
$$
\lim_{t\rightarrow +\infty}(x(t),s(t)) =\Big{(}\frac{b(0)}{n}a,0\Big{)}. 
$$
\end{prop}
\begin{proof} 
Observe that the mapping $t\mapsto s(t)$ decreases over $\R_+$, and that $s\geq 0$. Thus, $s(\cdot)$ necessarily converges to some value $\bar s$. By Barbalat's Lemma, 
$\lim_{t\rightarrow +\infty} \dot{s}(t)$ exists and is zero. Suppose now by contradiction that  $\bar s>0$. 
It follows that $\sum_{j=1}^n x_j(t) \rightarrow b(0)-\bar s$ and that $b(0)-\bar s>0$ because $x(0)\in \mathcal{D}$. Hence, there is $t_0\geq 0$ 
such that for every $t\geq t_0$, one has
$$
\dot{s}(t)=-\sum_{j=1}^n \mu_j(s(t))x_j(t)\leq - \upsilon \sum_{j=1}^n x_j(t)\leq -\upsilon \frac{ b(0)-\bar s}{2},
$$
where $\upsilon:=\frac{1}{2}\min_{1 \leq j \leq n}(\mu_j(\bar s))$. We have thus obtained a contradiction with the fact that 
$\dot{s}(t)\rightarrow 0$ as $t\rightarrow +\infty$.  
Let us now come back to  \eqref{tmp1}
which is a non-autonomous perturbation of the linear system 
\begin{equation}{\label{tmp2}}
\dot{x}=\eps Tx. 
\end{equation}
In order to apply the theory of asymptotically autonomous system \cite{Thieme92}, we need to rewrite \eqref{tmp2} in the orthogonal of $\R a$ in such 
a way that the corresponding autonomous dynamics possesses a unique globally asymptotically stable equilibrium (this is not the case with \eqref{tmp2} since zero is an eigenvalue of $T$).  
Doing so, we know that there exists an invertible matrix $P\in \R^{n\times n}$ such that 
$P^{-1}TP=D$ where $D:=\mathrm{diag}(0,\alpha_2,...,\alpha_n)$ with $\alpha_i<0$ for $2 \leq i \leq n$.  In addition, without any loss of generality, we may assume that the first column of $P$ is exactly equal to the vector $a$ (that is collinear to the Perron vector of $T$), and we also set $\tilde D:=\mathrm{diag}(\alpha_2,...,\alpha_n)\in \R^{(n-1)\times (n-1)}$. Multiplying \eqref{tmp1} on the left by $P^{-1}$ then gives\footnote{Given $v\in \R^n$, the notation 
$v_{-1}$ indicates that $v=(v_1,v_{-1})$ with $v_{-1}\in \R^{n-1}$ ($v_{-1}$ is the vector obtained from $v$ by removing the first component).}
\begin{equation}{\label{tmp3}}
\left|\begin{array}{cl}
\dot{y}&=(P^{-1} M(s(t))x)_1,\\
\dot{z}&=(P^{-1} M(s(t))x)_{-1} + \eps \tilde D z,
\end{array}\right.
\end{equation}
where $x=P\left(\begin{array}{c}y\\z\end{array}\right)$.
Next, the ODE satisfied by $z$ can be rewritten
$$
\dot{z}=F(t,z)+\eps \tilde D z,
$$
where $F:\R_+ \times \R^{n-1}\rightarrow \R^{n-1}$ is defined by
$$F(t,z):=\left( P^{-1} M(s(t))P\left(\begin{array}{c}y(t)\\ z\end{array}\right)\right)_{-1}.$$ 
Since $s(t)\rightarrow 0$ when $t\rightarrow +\infty$ and $y(\cdot)$ is bounded, the preceding system is a non-autonomous perturbation of 
 the linear system
\begin{equation}{\label{tmp4}}
\dot{z}=\eps \tilde D z. 
\end{equation}
Now, one has 
$F(t,z)\rightarrow 0$ when $t\rightarrow +\infty$ uniformly locally w.r.t.~$z$ and observe that every solution to \eqref{tmp4} converges to zero. We deduce from \cite{Thieme92} that every solution $(y(\cdot),z(\cdot))$ 
to \eqref{tmp3} is such that $z(t)\rightarrow 0$ when $t\rightarrow +\infty$. Coming back to the original variable $x$, the solution $x(\cdot)$ can be written
$$
x(t)=y(t)a+o(1).
$$
To conclude, observe that $t\mapsto b(t)$ is constant. Thus, for every $t\geq 0$, 
$$
b(0)=b(t)=\sum_{j=1}^n x_j(t)+s(t)=n y(t)  + o(1).
$$
Hence, $y(t)\rightarrow b(0)/n$ as $t\rightarrow +\infty$, which ends the proof. 
\end{proof}
%
\begin{rem}
This proposition shows that if $\eps>0$ and $u=0$, then, every species concentration converges to the same value $\frac{b(0)}{n}$ as $t\rightarrow +\infty$. 
In that case, any solution to \eqref{sys1} converges to the point $\Big{(}\frac{b(0)}{n}a,0 \Big{)}$ that depends on the initial condition. 
\end{rem}
\subsection{Global stability for $\eps>0$ and $u$ small enough}
Let us first recall Corollary 2.3 of \cite{SW99} which is a fundamental result about global stability of a perturbed steady-state. 
Let $k \geq 1$, and $G$, $U$ two closed subsets of $\R^n$ and $\R^k$ respectively.  
Consider a continuous function $g:G\times U \rightarrow \R^n$, $(x,u)\mapsto g(x,u)$ where $u\in U$ is a parameter. Suppose that $D_x g(x,u)$ exists and is continuous over $G \times U$ and that solutions $x(\cdot,x_0,u)$ to the Cauchy Problem
\begin{equation}{\label{sys-gen}}
\left|
\begin{array}{cl}
\dot{x}&=g(x,u), \vspace{0.1cm}\\
x(0)&=x^0,
\end{array}
\right.
\end{equation}
are unique and remain in $G$ for every time $t\geq 0$ and every $(x^0,u)\in G \times U$. 
\begin{thm}[\cite{SW99}]{\label{thm-SW}}
Let  $(x^\star,u^\star)\in G \times U$ be such that $x^\star\in \mathrm{Int}(G)$ and $g(x^\star,u^\star)=0$. Suppose that the matrix $D_x g(x^\star,u^\star)$ is Hurwitz and that $x^\star$ is globally attracting for solutions to \eqref{sys-gen} with $u=u^\star$. If there is a non-empty compact set $K\subset G$ such that for each $(x^0,u)\in G \times U$, 
$x(t,x^0,u)\in K$ for $t$ large enough, then, there are $r>0$ and a unique point $x^\star(u)\in G$ for every $u\in B(u^\star,r)$ such that 
$g(x^\star(u),u)=0$ and: 
\begin{equation}{\label{GAS-perturb}}
\forall u \in B(u^\star,r), \; \forall x^0\in G, \; x(t,x^0,u)\xrightarrow[t\rightarrow +\infty]{} x^\star(u). 
\end{equation}
\end{thm}
\begin{rem}
This result also applies if $x^\star$ is on the boundary of $G$ provided that the dynamics $g$ can be extended to a $C^1$ mapping  
in some convex neighborhood of $x^\star$ (see {\rm{\cite[Corollary 2.3]{SW99}}}). 
\end{rem}
The next lemma is based on Proposition \ref{deleen-prop}  (see \cite{deleen}) and it will be useful to prove Theorem \ref{GAS-thm}. 

\begin{lem}{\label{lem-robust}} 
For every $(\eps,u_1)\in \R_+^* \times (0,u_c(\eps))$, there is $\beta_\eps>0$ 
such that for every $u \in [0,u_1]$ and every initial condition in $\mathcal{D}$, the unique solution of \eqref{sys1} 
associated with this initial condition satisfies 
$$
\liminf_{t\rightarrow +\infty} \sum_{j=1}^n x_j(t) \geq \beta_\eps.
$$
\end{lem}
\begin{proof} 
Let $(\eps,u_1)\in \R_+^* \times (0,u_c(\eps))$. Given $u \in[0,u_1]$, the Perron vector $v^{\eps,u}$ associated with the greatest eigenvalue of the matrix $B(1,u,\eps)$ is the unique solution to the system 
$$
\left\{
\begin{array}{rl}
B(1,u,\eps) w-\lambda(B(1,u,\eps))w&=0,\\
\|w\|-1&=0.
\end{array}
\right.
$$
Now, consider the $C^1$-mapping 
$$
\tilde \psi:(u,w)\in[0,u_1] \times \R^n \mapsto \tilde \psi(u,w):=\Big{(} B(1,u,\eps) w-\lambda(B(1,u,\eps))w,\frac{\|w\|^2-1}{2}\Big{)}\in \R^{n+1}.
$$
Its partial derivative w.r.t.~$w$ at the point $(u,v^{\eps,u})\in [0,u_1]\times \R^n$ is given by: 
$$
\frac{\partial \tilde \psi}{\partial w}(u,v^{\eps,u})=\Big{(}B(1,u,\eps)-\lambda(B(1,u,\eps)) I_n,v^{\eps,u} \Big{)}.
$$
Hence,  if $w$ is in the kernel of $\frac{\partial \tilde \psi}{\partial w}(u,v^{\eps,u})$, it satisfies
$B(1,u,\eps)w=\lambda(B(1,u,\eps))w$ and $v^{\eps,u}\cdot w=0$ (here, $\cdot$ is the scalar product in $\R^n$). The first equality implies that there is 
$\nu \in \R$ such that $w=\nu v^{\eps,u}$. Using the second equality, we find that $\nu=0$, thus $w=0$ and $\frac{\partial \tilde \psi}{\partial w}(u,v^{\eps,u})$  
is invertible. Thanks to the implicit function theorem, we obtain that way that $u \mapsto v^{\eps,u}$ is locally continuous around every $u\in [0,u_1]$, thus it is continuous over $[0,u_1]$.
Now, Proposition \ref{deleen-prop} of \cite{deleen} implies that
$$
\liminf_{t \rightarrow +\infty} \sum_{j=1}^n x_j(t)\geq c  \min_{1 \leq j \leq n} v_j^{u,\eps}.
$$
Since the mapping $u\mapsto v^{\eps,u}$ is continuous over $[0,u_1]$, so is $u \mapsto \min_{1 \leq j \leq n} v_j^{u,\eps}$, hence,
$$
\liminf_{t \rightarrow +\infty} \sum_{j=1}^n x_j(t)\geq \beta_\eps:=c  \min_{u \in [0,u_1]}\min_{1 \leq j \leq n} v^{u,\eps}_j. 
$$
Because $u \mapsto \min_{1 \leq j \leq n} v_j^{u,\eps}$ is positive and continuous over $[0,u_1 ]$, we get that $\beta_\eps>0$ which ends the proof. 
\end{proof}
We now give our main result about  the global stability of the steady-state $E_{\eps,u}$ when $\eps>0$ is fixed and $u$ is with small enough values. 
\begin{thm}{\label{GAS-thm}}
For every $\eps>0$, there is $u_{s}(\eps)\in (0,u_c(\eps)]$ such that for every $u\in (0,u_{s}(\eps))$, the steady-state $E_{\eps,u}$ is globally asymptotically stable. 
\end{thm}
\begin{proof}
First, Proposition \ref{prop-ineg} implies that for every $u>0$ such that  $u<u_c(\eps)$, then one has $\lambda(B(1,u,\eps))>0$. Therefore, for every $u \in (0,u_c(\eps))$, the point $E_{\eps,u}$ is the unique locally asymptotically stable point of \eqref{sys1} in $\mathcal{D}$. We start by proving the result for those initial conditions that are in the set $\Delta$. Fix $u_1\in (0,u_c(\eps))$. The dynamics of $x$ can be then written $\dot{x}=g(x,u)$ where
$g:G\times U \rightarrow \R^n$ is defined as
$$
g(x,u):=B\Big{(}1-\sum_{j=1}^n x_j,u,\eps\Big{)}x,
$$
with $G:=\mathcal{D}'$ (recall \eqref{setF}) and $U:=[0,u_1]$.  
We set $x^\star:=\frac{a}{n}$, $u^\star=0$. We are then in a position to verify the hypotheses of Theorem \ref{thm-SW}: 
\begin{itemize}
\item[$\bullet$] At $(x^\star,u^\star)$, one has $g(x^\star,u^\star)=0$ and $x^\star\in \partial G$ since $\sum_{j=1}^n x^\star_j=1$ ; 
\item[$\bullet$] The Jacobian matrix of $g$ w.r.t.~$x$ at $(x^\star,u^\star)$ writes
$$
D_x g(x^\star,u^\star)=B\Big{(}1-\sum_{j=1}^nx_j^\star,u^\star,\eps\Big{)}-d_{0}\,  a^\top =\eps T-d_{0}\, a^\top,
$$
where $d_{0}:=M'(0)x^\star$ 
(recall the proof of Proposition \ref{disjonction} (ii)). It is a rank one perturbation of $\eps T$. By using a similar argumentation as in the proof of Proposition \ref{disjonction} (ii), we can check that it is Hurwitz ;  
\item[$\bullet$] By extending $\mu_i$ as a $C^1$ function over $\R$, the dynamics $g$ can be extended to a $C^1$ function in 
$(G\cap B(x^\star,\delta))\times U$ for every $\delta>0$ ; 
\item The set $K:=\{x\in G \; ; \; \beta_\eps \leq \sum_{j=1}^n x_j\leq 1\}$ is compact, and according to Lemma \ref{lem-robust}, 
for every $u\in [0,u_1]$ and for $t$ large enough, one has $x(t)\in K$ for every solution to $\dot{x}=g(x,u)$. 
\end{itemize}
We can then apply Theorem \ref{thm-SW} which implies the existence of $u_0\in (0,u_1]$ such that for every $u\in [0,u_0]$, 
the point $x^{\eps,u}$ is GAS for the dynamics restricted to $\Delta$. 

We now consider initial conditions in $\mathcal{D}$ and let $u\in (0,u_0]$ be fixed. The first $n$ equations in \eqref{sys1} write
$$
\dot{x}=B\Big{(}b(t)-\sum_{j=1}^n x_j,u,\eps\Big{)}x.
$$
This system is a non-autonomous perturbation of the autonomous system $\dot{x}=g(x,u)$ since $b(t)\rightarrow 1$ when $t\rightarrow +\infty$. 
Using a similar argumentation as in the proof of Proposition \ref{dilutionratezero} (see \cite{Thieme92}), we deduce that for every initial condition $x^0\in [0,+\infty)^n\backslash \{0\}$, one has $x(t)\rightarrow x^{\eps,u}$ as $t\rightarrow +\infty$. Now, given some initial condition  $(x^0,s^0)\in \mathcal{D}$ for system \eqref{sys1}, one has 
$b(t)\rightarrow 1$ as $t\rightarrow +\infty$. So, one has $s(t)\rightarrow s^{\eps,u}$ as $t\rightarrow +\infty$ 
which shows that for every $u\in [0,u_0]$, then $E_{\eps,u}$ is GAS. 

We now argue that the same reasoning can be employed starting from the point $x^{\eps,u_0}$ (which is GAS) in place of $x^\star$. We  obtain that way the existence of $u'_0>u_0$ such that $E_{\eps,u}$ is GAS for every $u \in [u_0,u'_0]$. Repeating this argumentation, one can define 
$$
u_s(\eps):=\sup \{u_0 \in (0,u_c(\eps)) \; ; \; \forall  u \in [0,u_0],\; E_{\eps,u} \; \mathrm{is} \; \mathrm{GAS} \}.
$$
This concludes the proof. 
\end{proof}

Showing that $E_{\eps,u}$ is GAS for every $(\eps,u)\in \R_+^* \times (0,u_c(\eps))$ seems a difficult question that could deserve further investigations based on results of Section \ref{sec-persist} (Theorem \ref{thm-main-persist}). 
We can make the following observations:
\begin{itemize}
\item[$\bullet$] If $u_s(\eps)=u_c(\eps)$, then we have the desired property. 
But, at this step, 
we only know that $u_s(\eps) \leq u_c(\eps)$. 
If $u_s(\eps)<u_c(\eps)$, note that $E_{\eps,u}$ remains LAS for every $u\in [u_s(\eps),u_c(\eps))$, {\it{i.e.}}, no bifurcation occurs at $u=u_s(\eps)$.  So, one can wonder if in this setting, such a loss of global stability is possible or not. 
\item[$\bullet$] Another approach consists in showing that $E_{\eps,u}$ is GAS for every $u>0$ provided that $\eps>0$ is small enough using a similar result as in Lemma \ref{lem-robust}, and proceeding as in the proof of Theorem \ref{GAS-thm}. One should  prove 
that for every  $u\in (0,\bar \mu(1))$, there is a constant $\gamma'_u>0$ (that does not depend on $\eps$) such that
$$
\liminf_{t\rightarrow +\infty} x_{i_0}(t)\geq \gamma'_u,
$$
for every $\eps$ small enough and every solution of \eqref{sys1},  
where species $i_0$ wins the competition in absence of mutation. 
\end{itemize}
In the next table, we summarize  asymptotic properties about \eqref{sys1} that have been established in this paper 
(including also the case without mutation, under the hypotheses of Theorem 
\ref{CEP-thm}\footnote{As in Theorem \ref{CEP-thm}, we do not mention here the (non-generic) cases where the dilution rate $u$ would be such that $u=\mu_i(s)=\mu_j(s)$ for some indexes $i\not=j$ and $s\in (0,1)$. 
}). 
\begin{table}[!ht]
\begin{center}
\begin{tabular}{|c|c|c|}
\hline
\multicolumn{1}{|c|}{\backslashbox{$u$}{\vrule width 0pt height 1.25em$\eps$}}& $\eps=0$ & $\eps>0$
\\
\hline
$u=0$ & Convergence into $\hat \Delta$ & Convergence to $\Big{(}\frac{b(0)}{n}a,0\Big{)}$
\\
\hline
$0<u<u_{s}(\eps)$ & $E_{i_0}$ GAS in $\mathcal{D}_{i_0}$ &  $E_{\eps,u}$ GAS in $\mathcal{D}$
\\
\hline 
$u_{s}(\eps)\leq u<u_c(\eps)$ & $E_{i_0}$  GAS in $\mathcal{D}_{i_0}$ & $E_{\eps,u}$ LAS
\\
\hline  
$u=u_c(\eps)$ & $E_{wo}$ S in $\R_+^n \times [0,1]$ &  $E_{wo}$ S in $\R_+^n \times [0,1]$
\\
\hline
$u>u_c(\eps)$ & $E_{wo}$ GAS in $\R_+^n \times [0,1]$ & $E_{wo}$ GAS in $\R_+^n \times [0,1]$
\\
\hline
\end{tabular}
\end{center}\caption[Table]{Summary of asymptotic properties of \eqref{sys1}. Here, 
$\hat \Delta:=\{(x,s)\in \Delta \; ; \; s=0\}$, 
$\mathcal{D}_{i_0}:=\mathcal{E}_{i_0}\times [0,1]$,  
$\mathcal{E}_{i_0}:=\{x\in \R_+^n \; ; \; x_{i_0}\not=0 \}$ and 
$\mathcal{D}=(\R_+^n \backslash \{0\}) \times [0,1]$. The abbreviations $S$, LAS, and GAS stand respectively for stable, locally asymptotically stable, globally asymptotically stable.}
\label{table-resume}
\end{table}

\section{Conclusion and perspectives}
In this paper, we could extend some results of \cite{deleen} showing that the coexistence steady-state of \eqref{sys1} is always LAS and in particular GAS 
provided that the dilution rate is small enough (assuming only that kinetics are of Monod type). 
Let us emphasize that in contrast with the chemostat system, mutation implies coexistence, {\it{i.e.}}, each species is present asymptotically.  
Future works could investigate global stability via a Lyapunov approach  at least for $\eps>0$ small enough taking into account the knowledge of a Lyapunov function for $\eps=0$. 
Asymptotic stability properties could be also addressed with more complicated mutation terms such as in \cite{Arkin,lobry}. 
As well, most properties proved in this paper are still valid if the kinetics are only increasing, hence, one can wonder if such properties remain valid with more sophisticated growth functions such as Haldane's kinetics.  
Finally, it could be also interesting to study continuous models describing the growth of a population structured by a phenotypical trait living in a limited substrate environment (see \cite{Perthame}). 
\section*{Acknowledgment}
This research benefited from the support of Avignon Universit\'e (AAP Agro\&Sciences) and from the support of the FMJH Program
PGMO and from the support to this program from EDF-THALES-ORANGE. 
The authors would also like to thank Francis Mairet, Pedro Gajardo, and Fr\'ed\'eric Mazenc for helpful discussions about Lyapunov functions. The authors are  grateful to P. De Leenheer and A. Rapaport for fruitful exchanges.

\end{document}